\documentclass[10pt,a4paper]{article}
\usepackage[T1]{fontenc}
\usepackage[utf8]{inputenc}
\pdfoutput=1
\usepackage{lmodern}
\usepackage{textcomp}
\usepackage{microtype}
\usepackage[DIV=10,BCOR=5mm,headinclude=false,footinclude=false]{typearea}
\usepackage[font=small,labelfont=bf]{caption}
\usepackage{titlesec}
\usepackage{tocloft}
\usepackage{amsmath,amssymb}
\usepackage{enumerate} 
\usepackage{amsfonts}
\usepackage{graphicx}
\usepackage{amsthm}
\usepackage{yfonts}
\usepackage{tikz-cd}
\usepackage{amssymb}
\usepackage{mathrsfs}
\usepackage[parfill]{parskip}
\usepackage{hyperref}
\usepackage[nameinlink]{cleveref}
\usepackage{MnSymbol }
\usepackage{verbatim}
\usepackage{array}

\newtheorem{theorem}{Theorem}

\newtheorem{corollary}{Corollary}[section]
\newtheorem{lemma}{Lemma}
\newtheorem{prop}{Proposition}[section]

\theoremstyle{remark}

\newcommand{\C}{\mathbb{C}}
\newcommand{\Ff}{\mathbb{F}}

\newcommand{\Gl}{\mathbf{GL}}

\newcommand{\ra}{\rightarrow}
\newcommand{\sra}{\twoheadrightarrow}
\newcommand{\hra}{\hookrightarrow}
\newcommand{\Rep}{\mathrm{Rep}}
\newcommand{\id}{\mathrm{Ind}}

\newcommand{\abs}{\nu}
\newcommand{\NN}{\mathbb{N}}
\renewcommand*{\det}{\qopname\relax o{det}}
\newcommand{\ZZ}{\mathbb{Z}}
\newcommand{\mvw}{^{\mathrm{MVW}}}
\newcommand{\irr}{\mathfrak{Irr}}
\newcommand{\ain}[3]{#1\in\{#2,\ldots,#3\}}
\newcommand{\ho}{\mathrm{Hom}}
\newcommand{\soc}{\mathrm{soc}}
\newcommand{\cosoc}{\mathrm{cosoc}}
\newcommand{\Sp}{\mathbf{Sp}}
\renewcommand{\sp}{\mathrm{Sp}}
\renewcommand{\mp}{\mathrm{Mp}}
\newcommand{\De}{\Delta}
\newcommand{\Z}{\mathrm{Z}}
\newcommand{\op}{\overline{P}}
\newcommand{\cc}{^{\mathfrak{c}}}

\renewcommand{\L}{\mathrm{L}}
\newcommand\restr[2]{{\left.\kern-\nulldelimiterspace #1\vphantom{\big|}\right|_{#2}}}
\newcommand{\dr}{\mathcal{D}_{\rho,r}}
\newcommand{\dl}{\mathcal{D}_{\rho,l}}
\newcommand{\dd}{\mathcal{D}_\rho}
\newcommand{\Dr}{\mathcal{D}_{\rho,r}^\mathrm{max}}
\newcommand{\Dl}{\mathcal{D}_{\rho,l}^\mathrm{max}}
\newcommand{\Dd}{\mathcal{D}_{\rho}^\mathrm{max}}
\newcommand{\drm}{d_{\rho,r,\mathrm{max}}}
\newcommand{\dm}{d_{\rho,\mathrm{max}}}
\newcommand{\dlm}{d_{\rho,l,\mathrm{max}}}

\renewcommand{\abs}[2]{\nu_{#1}^{#2}}
\newcommand{\LL}{\mathcal{L}}
\newcommand{\M}{\mathbf{Mat}}
\newcommand{\rk}{\mathrm{rk}}
\newcommand{\bX}{\mathbf{X}}
\newcommand{\bG}{\mathbf{G}}
\newcommand{\bB}{\mathbf{B}}
\newcommand{\bP}{\mathbf{P}}
\newcommand{\bL}{\mathbf{L}}
\newcommand{\bN}{\mathbf{N}}
\newcommand{\bY}{\mathbf{Y}}
\newcommand{\bD}{\mathbf{D}}
\newcommand{\oo}{\mathbb{O}}
\newcommand{\bon}{\mathbf{1}}
\newcommand{\Sg}{\mathcal{S}\mathbf{Gr}}
\newcommand{\sg}{\mathcal{S}\mathrm{Gr}}
\newcommand{\bH}{\mathbf{H}}
\renewcommand{\AA}{\mathbf{A}}
\newcommand{\bS}{\mathbf{S}}
\newcommand{\bT}{\mathbf{T}}
\newcommand{\bs}{\backslash}
\newcommand{\con}{\mathrm{Cone}}
\newcommand{\curC}{\mathcal{C}}
\newcommand{\curF}{\mathcal{F}}
\newcommand{\bU}{\mathbf{U}}
\newcommand{\Glt}{\widetilde{\mathrm{GL}}}
\newcommand{\gl}{\mathrm{GL}}

\newcommand{\tp}{\widetilde{P}}
\newcommand{\sgt}{\widetilde{\mathcal{S}\mathrm{Gr}}}
\newcommand{\DD}{\mathbb{D}}

\newcommand{\cuu}{\mathfrak{C}^{sd}}
\newcommand{\cus}{\mathfrak{C}}
\newcommand{\bQ}{\mathbf{Q}}
\newcommand{\QQ}{\mathbb{Q}}
\newcommand{\im}{\mathrm{Im}}
\newcommand{\bM}{\mathbf{M}}
\newcommand{\nr}{\nu_\rho}
\newcommand{\Se}{\mathcal{S}\mathrm{eg}}
\newcommand{\Ms}{\mathcal{M}\mathrm{ult}}
\newcommand{\Irr}{\irr}
\newcommand{\fm}{\mathfrak{m}}

\newcommand{\Dde}{\mathcal{D}_{l,\De}^{max}}
\newcommand{\irs}{\irr^\square}
\date{}

\begin{document}
\title{The spectrum of the symplectic Grassmannian and $\M_{n,m}$.}
\author{Johannes Droschl}
\maketitle
\begin{abstract}
   Let $\bG$ be a reductive group and $\bX$ a spherical $\bG$-variety over a local non-archimedean field $\Ff$. We denote by $S(\bX(\Ff))$ the Schwartz-functions on $\bX(\Ff)$. In this paper we offer a new approach on how to obtain bounds on
   \[\dim_\C\ho_{\bG(\Ff)}(S(\bX(\Ff)),\pi)\]
   for an irreducible smooth representation $\pi$ of $\bG(\Ff)$. Our strategy builds on the theory of $\rho$-derivatives and the Local Structure Theorem for spherical varieties. Currently, we focus on the case of the symplectic Grassmannian and the space of matrices. In particular, we obtain a new proof of Howe duality in type II as well as an explicit description of the local Miyawaki-liftings in the Hilbert-Siegel case. Furthermore, we manage to extend previous results of the author regarding the conservation relation in the theta correspondence to metaplectic covers of symplectic groups.
   Finally, we use our new proof of Howe duality in type II to relate the order of the poles of Godement-Jacquet $L$-functions to the geometry of the space of matrices and the order of poles of certain intertwining operators.
\end{abstract}
\section{Introduction}
Let $\bG$ be a quasi-split, connected, reductive group over a non-archimedean local field $\Ff$ with Borel subgroup $\bB$ and $\bX$ a spherical $\bG$-variety, \emph{i.e.} a normal variety over $\Ff$ which admits an open and dense $\bB$-orbit. In recent years the formulation of a relative Langlands program, \emph{cf.} \cite{SakVen}, \cite{Sak23}, \cite{Sak08}, and more generally a relative Langlands duality proposed in \cite{BenSakVen}, extended the classical local Langlands correspondence between irreducible, smooth complex representations of $\bG(\Ff)$ and the local $L$-parameters of $\mathrm{Gal}(\overline{\Ff}/\Ff)$. 
These extensions require a precise understanding of the \emph{spectrum} $\irr_\bX(\bG(\Ff))$ of $\bX$, \emph{i.e.} the set of irreducible representations $\pi$ of $\bG(\Ff)$ admitting a non-zero morphism $S(\bX(\Ff))\ra\pi$, where $S(\bX(\Ff))$ denotes the locally constant, compactly supported functions on $\bX(\Ff)$. The most interesting cases are those where $\bX$ is \emph{multiplicity free}, \emph{i.e.} when the above morphism is, up to a scalar, unique for all irreducible representations. 
The theory of spherical varieties offers itself very nicely for such considerations, in part because of the combinatorial classification due to \cite{LunVus83} and the work of \cite{Kno90}, \cite{Kno96}.

In this paper we will explain how one can compute the spectrum and prove multiplicity-freeness in two specific cases, namely where $\bG=\Gl_n\times \Gl_m$ and $\bX=\M_{n,m}$, the space of $n\times m$-matrices, as well as $\bG=\Sp_n\times \Sp_m$ and $\bX=\Sg_{n,m}\coloneq\bP_{n+m}\bs\Sp_{n+m}$, where $\bP_{n+m}$ is the Siegel-parabolic of the symplectic group $\Sp_{n+m}$. Note that $\Sg_{n,m}$ is the space of $n+m$-dimensional isotropic subspaces of $\AA^{2n+2m}$ equipped with the standard symplectic structure and from now on we assume that $n\le m$.
Our approach rests on two pillars: Firstly, the existence of crystal-like structures in the form of $\rho$-derivatives, \emph{cf.} \cite{Mder}, \cite{Jader}, \cite{AtoMin}, and secondly the Local Structure Theorem for spherical varieties in combination with the Jacquet-functor and Frobenius reciprocity. Due to the general nature of our arguments we hope that in the future we can extend our results to a larger class of spherical varieties. The idea of approaching these problems from such an angle was already present in the work of the author in \cite{DroKudRa}, \cite{DroHow}, although in a somewhat hidden manner. Finally, we also address the case of metaplectic covers of the symplectic groups in the context of \cite{DroKudRa}, and thus obtain a new proof of the conservation relation in the type I theta correspondence in the corresponding cases.

Still, the results of this paper have two new applications. Namely, we obtain a new proof of Howe-duality in type II, originally due to \cite{Minguez}, and which shows, together with \cite{DroKudRa}, \cite{DroHow}, that all versions of Howe-duality follow from the general arguments laid out below. Secondly, we give an explicit description of the local Miyawaki-lifts in the Hilbert-Siegel case as introduced in \cite{Ato19}. Let us now explain some of the above ideas in more detail.

Throughout the exposition, we will always assume that all orbits are defined over $\Ff$. We will write for any variety $\bY$ over $\Ff$, $Y\coloneq\bY(\Ff)$ and order the, necessarily finitely many, $\bG$-orbits of $\bX$ via the closure-order, \emph{i.e.} we write for two orbits $\bY,\bY'$ $\bY'\le\bY$ if $Y'\subseteq \overline{Y}$. We call $\bX$ \emph{ordered} if the closure-order is a total order and note that each $\bG$-orbit $\bY$ is again a spherical variety. If $\pi$ is an irreducible smooth representation of $G$, we say $\bX$ has the lifting property with respect to $\pi$ if the following holds:
If $f\colon S(X)\ra\pi$ is a non-zero morphism and $\pi$ is in spectrum of some $\bG$-orbit $\bY$, then $f$ vanishes on all $\phi\in S(X)$ whose support is contained in \[\bigcup_{\bY'>\bY}Y'.\] The following is then an easy consequence of the definitions.
\begin{lemma}
    If $\bX$ is ordered, has the lifting property for all irreducible smooth representations and each orbit $\bY$ is multiplicity-free, then $\bX$ is multiplicity-free.
\end{lemma}
Note that in our specific cases the orbits are given by $\M_{n,m}^r$, the matrices of rank $r$, $\ain{r}{0}{n}$, respectively by $\Sg_{n,m}^r$, the isotropic subspaces $U$ such that the dimension of $U$ intersected with the subspaces spanned by the first $n$-coordinates equals $r$, $\ain{r}{0}{n}$. 
For the following results, we will have to introduce some notation.
Denote by $\lvert-\lvert$ the absolute value of $\Ff$ and by $\nu_n\coloneq \lvert\det_n\lvert\colon\gl_n\ra\C^\times$. We let $\irr(G)$ be the set of irreducible smooth complex representations of $G$. Finally, we denote for any smooth finite length representation $\pi$ of $G$ by $\soc(\pi)$ its socle, \emph{i.e.} its maximal semi-simple subrepresentation.
We use the standard notation for parabolic induction for general linear groups and symplectic groups, \emph{i.e.} we let $\bP_{n_1,n_2}$ be the parabolic subgroup of $\Gl_{n_1+n_2}$ containing the upper triangular matrices with Levi component $\Gl_{n_1}\times \Gl_{n_2}$ and for $\pi_1\otimes \pi_2\in \irr(\gl_{n_1}\times \gl_{n_2})$, we write $\pi_1\times\pi_2$ for the normalized parabolic induction of $\pi_1\otimes\pi_2$ to $\gl_{n_1+n_2}$. Similarly, we let $\bP_{n_1}$ be the standard parabolic subgroup of $\Sp_{n_1+n_2}$ with Levi component $\Gl_{n_1}\times \Sp_{n_2}$ and for $\pi_1\otimes \pi_2\in \irr(\gl_{n_1}\times \sp_{n_2})$, we write $\pi_1\rtimes\pi_2$ for the normalized parabolic induction of $\pi_1\otimes\pi_2$ to $\sp_{n_1+n_2}$. For any parabolic subgroup $\bP$ we denote by $\overline{\bP}$ the opposite parabolic subgroup and denote by $\delta_P$ its modular character.
\begin{lemma}[\cite{Jader},\cite{Mder}]
    Let $\rho$ be a cuspidal representation of $\gl_m$ and $\pi,\pi'\in\irr(\gl_n),\,\sigma\in\irr(\gl_{n+m})$. Then $\pi\times\rho$ and $\rho\times\pi$ have an irreducible socle. If $\sigma\hra\pi\times\rho$ and $\sigma\hra\pi'\times\rho$, then $\pi\cong \pi'$ and $\pi$ is called the \emph{right-$\rho$-derivative} of $\sigma$.
    If $\sigma\hra\rho\times\pi$ and $\sigma\hra\rho\times\pi'$, then $\pi\cong \pi'$ and $\pi$ is called the \emph{left-$\rho$-derivative} of $\sigma$.
\end{lemma}
\begin{lemma}{\cite{AtoMin}}
    Let $\rho$ be a cuspidal representation of $\gl_m$ such that $\rho$ is not self dual and $\pi,\pi'\in\irr(\sp_n),\,\sigma\in\irr(\sp_{n+m})$. Then $\rho\rtimes\pi$ has an irreducible socle. If $\sigma\hra\rho\rtimes\pi$ and $\sigma\hra\rho\rtimes\pi'$, then $\pi\cong \pi'$ and $\pi$ is called the \emph{$\rho$-derivative} of $\sigma$.
\end{lemma}
Using the theory of $\rho$-derivatives one can then show the following.
\begin{lemma}
    Let $\ain{r}{0}{n}$. The spaces $\M_{n,m}^r$ and $\Sg_{n,m}^r$ are multiplicity-free.
    Moreover, \[\irr_{\M_{n,m}^r}(\gl_n\times \gl_m)=\{\pi\otimes\pi':\pi\otimes \pi'\hra \abs{n-r}{\frac{r}{2}}\times \tau\otimes\nu^{\frac{m-n}{2}}\tau^\lor \times\abs{m-r}{-\frac{r}{2}}: \tau\in \irr(\gl_r)\}\]
    and
        \[\irr_{\Sg_{n,m}^r}(\sp_n\times \sp_m)=\{\pi\otimes\pi':\pi\otimes\pi'\hra \abs{r}{-\frac{r+n-m}{2}}\rtimes \tau\otimes \abs{m-n+r}{-\frac{r}{2}}\rtimes \tau^\lor,\,  \tau\in \irr(\sp_{n-r})\}.\]
\end{lemma}
By \cite[Lemma 2.8]{MinLa18}, $\pi'$ is uniquely determined by $\pi$ in the case of $\M_{n,m}^r$, however in the case of $\Sg_{n,m}^r$, there might exist two non-isomorphic choices for $\pi'$, see \Cref{S:Deder}.
Let us state the main technical lemma of that section.
\begin{lemma}
    Let $m,n\in\NN$, $\pi,\pi'\in \irr(\sp_n)$ and $s\in \C$. 
    \begin{enumerate}
        \item The socle of $\nu_m^s\rtimes\pi$ is either of length $1$ or $2$.
        \item If the socle is of length $2$, its two summands are non-isomorphic. 
        \item The socle is irreducible if $s>0$, or if $s+\frac{m-1}{2}<0$, or if $s\le 0$ and $2s\notin \ZZ$.
        \item Let $\sigma$ be an irreducible subrepresentation of $\nu_m^s\rtimes\pi$ and assume that $\sigma\hra \nu_m^s\rtimes\pi'$. Then $\pi\cong \pi'$.
        \item The cosocle of the image of the normalized intertwining operator 
    \[I:\nu_m^{-s}\rtimes\pi\ra \nu_m^s\rtimes\pi\] is isomorphic to the socle of $\nu_m^s\rtimes\pi$.
    \end{enumerate}
\end{lemma}
To prove that $\Sg_{n,m}$ and $\M_{n,m}$ have the lifting property, we use the following theorem, known as the Local Structure Theorem.
For $\bY$ a $\bG$-orbit of $\bX$ we denote by
\[\bX^\circ_{\bY}\coloneq \{x\in \bX: \bY\subseteq \overline{\bB x}\}.\]
Let $\bP_\bY$ be the stabilizer of $\bX_\bY$, which is a parabolic subgroup with Levi-subgroup $\bL_\bY$ and unipotent part $\bN_\bY$.
\begin{theorem}[{Local Structure Theorem, \cite{LunVus83}}]
    Then there exists an $\bL$-spherical, affine variety $\bS_\bY\subseteq \bX^\circ_{\bY}$ over $\Ff$ such that 
    \[\bN_\bY\times \bS_\bY\ra \bX^\circ_{\bY}\] is a $\bP_\bY$-equivariant isomorphism and $\bX^\circ_{\bY}\cap \bY$ consists of the unique open and dense $\bB$-orbit in $\bY$.  
\end{theorem}
In particular, we can formulate the following inductive approach to show the lifting property.
Assume one can find for any $\pi\in\irr(G)$ a suitable orbit $\bY$ and $\sigma\in\irr(L_Y)$ such that 
\[\pi\hra \id_{P_Y}^{G}(\sigma)\]
and for any non-zero morphism
\[S(\bX)\ra\pi\] the composition 
\[S(\bX^\circ_{\bY})\hra S(\bX)\ra\pi\hra\id_{P}^{G}(\sigma)\]
is non-zero. Then we can apply Frobenius reciprocity and obtain an injection
\[\ho_{G}(S(\bX),\pi)\subseteq \ho_{L_Y}(S(\bS_\bY),\delta_{P_Y}^{-\frac{1}{2}}\sigma)\]
and hope that by inductive argument we deduce the lifting property for $\pi$ from the lifting property of $\sigma$. 

In our specific case, the above takes for a suitable choice of a Borel-subgroup the following form.
\begin{lemma}
    Let $\ain{r}{0}{n}$. If $\bY=\M_{n,m}^r$, then $\bP_\bY= \bP_{r,n-r}\times \overline{\bP_{r,m-r}}$ and
    \[\bS_{\M_{n,m}^r}=\M_{n-r,m-r}\times \Gl_r.\]
    If $\bY=\Sg_{n,m}^r$, then $\bP_\bY= \bP_{n-r}\times \overline{\bP_{m-r}}$ and \[\bS_{\Sg_{n,m}^r}=\M_{n-r,m-r}\times \Sp_r.\]
\end{lemma}
This allows us to implement the above-suggested strategy, thereby enabeling us to show the following.
For $\pi\in\irr(\gl_n)$, let $\ain{r}{0}{n}$ be minimal such that there exists $\tau\in\irr(\gl_r)$ with $\pi\hra\abs{n-r}{\frac{r}{2}}\times\tau$. Then \[ \mathfrak{T}_n^m(\pi)\coloneq\mathrm{soc}(\tau^\lor\nu^{\frac{m-n}{2}}\times \abs{m-r}{-\frac{r}{2}}).\]

Let $\pi\in \irr(\sp_n)$ and let $\ain{r}{0}{n}$ be maximal such that there exists $\tau\in\irr(\sp_{n-r})$ with $\pi\hra\abs{r}{-\frac{m-n+r}{2}}\rtimes\tau$. Then \[\mathfrak{T}_n^m(\pi)\coloneq\mathrm{soc}(\abs{m-n+r}{-\frac{r}{2}}\rtimes\tau).\]
\begin{theorem}
    The space $\M_{n,m}$ is multiplicity-free, has the lifting property for all irreducible representations and \[\irr_{\M_{n,m}}(\gl_n\times \gl_m)=\{\pi\otimes \mathfrak{T}_n^m(\pi):\, \pi\in \irr(\gl_n)\}.\]
    The space $\Sg_{n,m}$ is multiplicity-free, has the lifting property for all irreducible representations  and \[\irr_{\Sg_{m,n}}(\sp_n\times \sp_m)=\{\pi\otimes\pi':\pi\otimes\pi'\subseteq \pi\otimes \mathfrak{T}_n^m(\pi),\,  \pi\in \irr(\sp_{n})\}.\]
\end{theorem}
We also formulate analogous statements for general linear groups over division algebras and metaplectic covers of both $\gl_n$ and $\sp_n$. In particular the case of $\mp_n$ gives an explicit description of local Miyawaki-liftings in the Hilbert-Siegel case, see \Cref{S:app}. Additionally, the case of the metaplectic cover of $\gl_n$ allows us to extend the arguments of \cite{DroKudRa} and \cite{DroHow} to metaplectic covers, see \cite[Remark 2.9.2]{DroKudRa}.

Finally, let us connect our results in the case $\M_{n,n}$ to the behavior of the Godement-Jacquet $L$-function $L(\pi,s)$ of a smooth, irreducible representation $\pi$ of $\gl_n$ and certain intertwining operators.
To state our theorem, we introduce the following classical definition. For $\pi_1,\pi_2$ smooth, irreducible representation of $\gl_{n_1}$ and $\gl_{n_2}$, $s\in \C$, we denote the intertwining operator by
\[I_{\pi_1,\pi_2}(s)\colon \pi_2\lvert-\lvert^s\times\pi_1\ra \pi_1\times \pi_2\lvert-\lvert^s,\] and the order of its pole at $s=0$ by $\Lambda(\pi_1,\pi_2)$.
\begin{theorem}
Let $\pi\in \irr(\gl_n)$. Then $L(\pi,s)$ has a pole at $s=-\frac{n-1}{2}$ if and only if
the, up to a scalar, unique map \[T_\pi\colon S(\M_{n,n})\ra \pi\otimes\pi^\lor\] vanishes on $S(\Gl_n)$. To be more precise,
let $r$ be maximal such that $T_\pi$ vanishes on \[S(\bigcup_{n\ge r'\ge r+1}\M_{n,n}^{r'}). \]Then there exists an, up to isomorphism, unique irreducible representation $\tau$ of $\gl_{r}$ such that $\pi$ is a subrepresentation of $\nu_{n-r}^{\frac{r}{2}}\times\tau$. Moreover, \[\mathrm{ord}_{s=-\frac{n-1}{2}}L(\pi,s)=\Lambda(\nu_{n-r}^{\frac{r}{2}},\tau)+1.\]
\end{theorem}
In particular, this allows us to recover the description of the $L$-factor in terms of its Langlands-parameter, as was already achieved in \cite{JL2}. As an interesting side note, our proof does not use the functional equation.
\subsection*{Acknowledgements}
I am grateful to Alberto M{\'i}nguez for his suggestion and encouragement to extend the results of my thesis to the cases discussed above. Our conversations have consistently been a great source of inspiration for my research.
This work has been supported by the projects PAT4628923 and PAT4832423 of the Austrian Science Fund (FWF).
\section{Preliminaries}
\counterwithin{theorem}{section}
\counterwithin{lemma}{section}
Let $\Ff$ be a non-archimedean field with absolute value $\lvert-\lvert$ and residue cardinality $q$ and let $\bG$ be a connected reductive group over $\Ff$. We fix a smooth additive character $\psi\colon\Ff\ra\C^\times$. If $\bX$ is a variety over $\Ff$, we will write usually $X=\bX(\Ff)$ and equip $X$ with the topology inherited by $\Ff$ as follows. If $\bX$ is affine, we note \[X=\ho(\Ff[\bX],\Ff),\]
where we equip the latter space with the compact-open topology and $\Ff[\bX]$ has the discrete topology. For non-affine $\bX$, we just choose an affine cover and equip $X$ with the corresponding topology.
A locally constant function on $X$ refers to a function $\phi\colon X\ra\C$ which is locally constant with respect to the aforementioned topology.
We denote by $\Rep(G)$ the category of smooth admissible representations of $G$ with complex coefficients. We write for injective respectively surjective maps in this category $\hra$ respectively $\sra$, and denote the image of a morphism $f$ by $\im(f)$.
If $\bH$ is a closed subgroup of $\bG$, we denote 
\[\#-\id_H^G\colon \Rep(H)\ra \Rep(G)\]
the functor of compactly supported induction.
If $\bP$ is a parabolic subgroup of $\bG$, with Levi-decomposition $\bL\bN$, where $\bN$ is the unipotent part, we denote by $\delta_P=\delta_N$ the modular character of $P$ and by \[\id_P^G\coloneq \#-\id_P^G(\delta_P\otimes-)\colon \Rep(L)\ra \Rep(G),\]
where $P$ acts on an $L$-representation via the canonical surjection $P\sra L$.
We also denote the Jacquet functor, \emph{i.e.} the $N$-coinvariants, by $\#-r_P\colon \Rep(G)\ra \Rep(L)$, and its normalized version by $r_N=r_P\coloneq\delta_P^{-\frac{1}{2}}\otimes\#-r_P$.
The functors $\id_P^G$ and $r_P$ are exact, preserve representations of finite length and $\id_P^G$ has as a left adjoint $r_P$ and as a right adjoint $\overline{r_P}\coloneq r_{\overline{P}}$. These facts are referred to as \emph{Frobenius-} respectively \emph{Bernstein reciprocity.} 

Let $\bP=\bL_{\bP}\bN_{\bP}$ and $\bQ=\bL_{\bQ}\bN_{\bQ}$ be two parabolic subgroups of $\bG$. We recall the description of the functor $r_{Q}\circ \id_P^G\colon \Rep(L_P)\ra \Rep(L_Q)$.  Define for $\mathfrak{w}\in \bP\bs \bG/\bQ$ the groups
 \[\bL_{\bP}'\coloneq \bL_{\bP}\cap w^{-1}\bL_{\bQ}w,\, \bL_{\bQ}'\coloneq w\bL_{\bP}'w^{-1},\,\bN_{\bQ}'\coloneq \bL_{\bP}\cap w^{-1}\bN_{\bQ}w,\, \bN_{\bP}'\coloneq \bL_{\bQ}\cap w \bN_{\bP}w^{-1},\]
where $w$ is a representative of $\mathfrak{w}$. Let \[\delta_1\coloneq \delta_{N_P}^{\frac{1}{2}}\cdot\delta_{N_P\cap w^{-1}Qw}^{-{\frac{1}{2}}},\, \delta_2\coloneq\delta_{N_Q}^{\frac{1}{2}}\cdot\delta_{N_Q\cap wPw^{-1}}^{-{\frac{1}{2}}}\] be characters of $L_P'$ respectively $L_Q'$. Finally, let $\mathrm{Ad}(w)\colon \Rep(L_P')\ra \Rep(L_Q')$ be the pullback by conjugation by $w$ and set $\delta\coloneq \delta_1w^{-1}(\delta_2)$.
Order the $\bQ$-orbits of $\bP\bs \bG$ as  $\mathfrak{w}_1,\ldots, \mathfrak{w}_l$ such that $\overline{{\mathfrak{w}_i}}=\bigcup_{i\le j}{\mathfrak{w}_j}.$

 Define for $\mathfrak{w}\in \bP\bs \bG/\bQ$ the functor \[F(\mathfrak{w})\coloneq \id_{M_P'N_P'}^{M_Q}\circ \mathrm{Ad}(w)\circ \delta\circ r_{N_Q'}\colon \Rep(M)\ra\Rep(N).\]
Then the following holds.
\begin{lemma}[\cite{Ber} Lemma 2.11]
The functor \[F\coloneq r_Q\circ \id_P^G\colon \Rep(M)\ra\Rep(N)\] has a filtration $0=F_0\subseteq F_1\subseteq F_2\subseteq\ldots\subseteq F_l=F$ with subquotients $F_{i-1}\bs F_i\cong F(\mathfrak{w}_i)$. 
\end{lemma}

We let $\irr(G)$ be the set of isomorphism classes of irreducible representations in $\Rep(G)$ and $K(G)$ be the freely generated group by $\irr(G)$. If $\pi\in \Rep(G)$ is of finite length we denote by $\soc(\pi)$ the maximal, semi-simple subrepresentation of $\pi$, and by $\cos(\pi)$ the maximal, semi-simple quotient of $\pi$. The representation $\pi$ is called cosocle- respectively socle-irreducible, if $\cos(\pi)$ respectively $\soc(\pi)$ is irreducible.

A cone $\curC$ in a $\QQ$-vector space $V$ is a subset which is stable under addition and multiplication under $\QQ_{\ge 0}$. A cone $\curC$ is called strictly convex if $\curC\cap -\curC$ contains no line and the dual cone is defined as $\curC^\lor\coloneq \{v^\lor\in V^\lor:v^\lor(\curC)\ge 0\}$. The faces of $\curC$ are given by the vanishing sets of $v^\lor\in\curC^\lor$ in $\curC$ and the interior $\curC^\circ$ of $\curC$ is given by the complement of all faces. Finally, the dimension of $\curC$ is the dimension of its linear span.

A partition $\alpha=(\alpha_1,\ldots,\alpha_k)$ of $\lvert\alpha\lvert\coloneq m\in \NN$ refers to a tuple of natural numbers which sum to $m$ and if $\alpha_1\ge \ldots\ge \alpha_k$, we called it ordered.
\subsection{The general linear group}\label{S:repgl}
Let $\DD$ be a central division algebra over $\Ff$ of degree $d$ and consider the group $\Gl_n'$ of invertible $n\times n$ matrices over $\DD$. We fix a minimal parabolic subgroup $\bB_n'$ consisting of upper triangular matrices. We will sometimes write $\irr_n\coloneq \irr(\gl_n')$ and $\irr=\bigcup_{n\in\NN}\irr_n$. Similarly we define $\Rep_n$ and $\Rep$. For $\pi\in \Rep_n$ we write $\deg(\pi)=n.$

In the following discussion we will omit the marking "$'$" if $\DD=\Ff$.
To a partition $(\alpha_1,\ldots,\alpha_k)$ of $n$ we associate the parabolic subgroup $\bP_\alpha'$ of $\Gl_n'$ containing $\bB_n'$, whose Levi-component consists of the block-diagonal matrices of the form $\Gl_{\alpha_1}'\times\ldots\times \Gl_{\alpha_k}'$. Writing $\overline{\alpha}\coloneq (\alpha_k,\ldots,\alpha_1)$, we recall that $\overline{\bP_\alpha'}$ is conjugated to $\bP_{\overline{\alpha}}$.

Parabolic induction and the Jacquet functor with respect to $\bP_\alpha'$ are then denoted as usual by
\[\pi_1\times\ldots\times\pi_k\coloneq \id_{P_\alpha'}^{\gl_n'}(\pi_1\otimes\ldots\otimes \pi_k)\text{ and } r_\alpha\coloneq r_{P_\alpha'}.\]
The product \[\times\colon \Rep_n\times \Rep_m'\ra \Rep_{n+m}\] is not commutative, however since parabolic induction is exact it induces a product on \[K(\gl_n')\times K(\gl_m')\ra K(\gl_{n+m}'),\] which turns out to be commutative.
In particular if $\pi\times\pi'$ is irreducible, $\pi\times\pi'\cong \pi'\times \pi$. For $k\in \NN$, we write \[\pi^{\times k}\coloneq \overbrace{\pi\times\ldots\times\pi}^{ k}.\]
Finally, we recall the cuspidal support of $\pi$, \emph{i.e.} the well-defined, formal sum $\rho_1+\ldots+\rho_k$, where $\rho_1\otimes\ldots\otimes \rho_k$ is a cuspidal representation such that $\pi\hra\rho_1\times\ldots\times\rho_k$.
For $\pi\in \Rep_n$ we denote $\pi\cc$ the representation obtained by twisting $g\mapsto {}^tg^{-1}$.
\begin{lemma}
    If $\pi$ is irreducible, $\pi\cc\cong\pi^\lor$. Moreover, for $\pi_1,\pi_2\in \Rep(\gl_{n_1}\times\gl_{n_2})$, \[(\pi_1\times\pi_2)\cc=\id_{\overline{P_{n_1,n_2}}}^{\gl_{n_1+n_2}}(\pi_1\cc\otimes\pi_2\cc).\]
\end{lemma}
We recall the following properties of induced representations, see \cite[Theorem 1.9, Theorem 4.1, Proposition 4.6 and Theorem 6.1]{Zel}.
Let $\xi\colon\gl_1\ra \C^\times$ be a smooth character. We denote by $\xi_n\colon\gl_n'\ra\C^\times$ the character $\xi(\det_n')$, where $\det_n'$ denotes the reduced norm of $\Gl_n'$. If $n=1$, we often write $\xi=\xi_1$. For $\pi\in \Rep_n$ we write $\chi\pi\coloneq\chi_n\otimes\pi\in\Rep_n$ and we denote the trivial representation in $\Rep_{\gl_n'}$ by $\bon_n$.
For a partition $\alpha=(\alpha_1,\ldots,\alpha_k)$, 
\[\delta_{P_\alpha}=\abs{\alpha_1}{q(\alpha)_1}\otimes\ldots\otimes \abs{\alpha_k}{q(\alpha)_k}, q(\alpha)_i\coloneq n-\sum_{j=1}^{i-1}2\alpha_j-\alpha_{i}\]
\subsubsection{Multisegments}
Let $a,b\in \ZZ$, $a\le b$, $\rho\in \cus$. To this datum we associate the segment $[a,b]_\rho\coloneq ([\rho \nr^a],\ldots,[\rho\nr^b])$. Note that $[a,b]_\rho=[c,d]_{\rho'}$ if and only if $\rho\nr^a\cong \rho'\nu_{\rho'}^c$ and $d-c=b-a$.
A segment $\De_1=[a,b]_\rho$ precedes a segment $\De_2=[c,d]_{\rho'}$, denoted by
$\De_1\prec\De_2$, if $\rho'\cong \rho\nr^x,\,x\in \ZZ$ and
\[a+1\le c+x\le b+1\le d+x.\]
We call $\De_1$ and $\De_2$ unlinked if neither $\De_1\prec\De_2$ nor $\De_2\prec\De_1$.
We set
\[l([a,b]_\rho)\coloneq b-a+1,\, \deg([a,b]_\rho))\coloneq l([a,b]_\rho)\deg(\rho).\]
The set of multisegments
is defined as $\Ms\coloneq \NN(\Se)$, the set of formal finite sums of segments. We extend the functions $l$ and $\deg$ linearly to $\Ms$.
For a multisegment $\De_1+\ldots+\De_k$, we say $(\De_1,\ldots,\De_k)$ is an arranged form if there exists no $i<j$ such that $\De_i\prec\De_j$.
\begin{theorem}[\cite{Zel}]\label{T:zel}
    There exists a bijection
    \[\Z\colon \Ms\ra\Irr\] satisfying the following properties.
    \begin{enumerate}
        \item $\Z([a,b]_\rho)=\soc(\rho\nr^a\times\ldots\times\rho\nr^b)$.
        \item For $m,n\in\NN,\,m+n=\deg([a,b]_\rho)$,
        \[r_{m,n}(\Z([a,b]_\rho))=0\] unless $\deg(\rho)\lvert m$, in which case, \[r_{m,n}(\Z([a,b]_\rho))=\Z([a,a+\frac{m}{\deg(\rho)}-1]_\rho)\otimes \Z([a+\frac{m}{\deg(\rho)},b]_\rho).\]
        \item If $(\De_1,\ldots,\De_k)$ is an arranged form of $\fm\in\Ms$, then
        \[\Z(\fm)=\soc(\Z(\De_1)\times\ldots\times \Z(\De_k))=\cosoc(\Z(\De_k)\times\ldots\times\Z(\De_1))\]
        and appears with multiplicity $1$ in $\Z(\De_1)\times\ldots\times \Z(\De_k)$.
        \item For $\fm\in\Ms$, $\deg(\Z(\fm))=\deg(\fm)$.
        \item If $\fm=\De_1+\ldots+\De_k$ is a multisegment such that the $\De_i$ are pairwise unlinked, then 
        \[\Z(\fm)=\Z(\De_1)\times\ldots\times \Z(\De_k).\]
        \item If $\fm'$ is a second multisegment, $\Z(\fm+\fm')$ appears in $\Z(\fm)\times \Z(\fm')$ with multiplicity $1$.
    \end{enumerate}
\end{theorem}
\subsubsection{Aubert-Zelevinsky involution}
We recall the involution 
\[(-)^*\colon \Ms\ra \Ms\] of \cite{Zel}, \cite{MoeglinetJ1986}, see also \cite[Theorem 1.7]{Aub}.
As usual, we write $\L(\fm)\coloneq \Z(\fm^*)$.
\begin{theorem}
    There exists an involution $(-)^*\colon \Ms\ra \Ms$ with the following properties.
    \begin{enumerate}
        \item For $\fm\in \Ms$, $\deg(\fm)=\deg(\fm^*)$.
        \item For $\rho\in \cus$, $[0,0]_\rho^*=[0,0]_\rho$.
        \item If $\fm_1,\fm_2,\fm_3\in \Ms$ with either $\fm_2$ or $\fm_3$ consisting of a segment of length $1$ and $\Z(\fm_1)\hra \Z(\fm_2)\times \Z(\fm_3)$, then
        \[\L(\fm_1)\hra \L(\fm_3)\times \L(\fm_2).\]
        \item In particular, the cuspidal support of $\Z(\fm)$ and $\L(\fm)$ is the same.
        \item For $[a,b]_\rho\in \Se$,
        \[[a,b]_\rho^*=[a,a]_\rho+\ldots+[b,b]_\rho.\]
        
    \end{enumerate}
\end{theorem}
\subsection{$\De$-derivatives for general linear groups}\label{S:deltader}
Following \cite[§7]{LapMin25}, we define for $\pi\in \Irr$ and $\De=[0,b]_\rho$ a segment, the two representations
$\Dde(\pi),\De(\pi)\in \Irr$ as the irreducible representation such that 
\begin{enumerate}
    \item $\pi\hra \De(\pi)\times \Dde(\pi)$.
    \item $\De(\pi)\cong \Z([a_1,b]_\rho)\times\ldots\times\Z([a_k,b]_\rho)$ with $\ain{a_i}{0}{b}$ for all $\ain{i}{1}{k}$.
    \item There exists no $\pi'\in \Irr_{m}$ for suitable $m$ and $\ain{a}{0}{b}$ such that \[r_{(b-a+1)\deg(\rho),m}(\Dde(\pi))\] contains $ \Z([a,b]_\rho)\otimes \pi'$ as an irreducible subquotient.
\end{enumerate}
We also allow here for the possibility that $\De(\pi)$ is the $0$-representation, in which case we call $\pi$ $\De$-reduced.
In the above situation, we write $l_\De(\pi)=k$.
\begin{lemma}[{\cite[Lemma 7.2]{LapMin25}}]\label{L:derL}
    The above properties define $\Dde(\pi),\De(\pi)\in \Irr$ up to isomorphism uniquely. Moreover, the map
    \[\Dde\times \De(-)\colon\Irr\ra\Irr\times \Irr\cup\{0\}\] is injective.
\end{lemma}
\begin{lemma}[{\cite[Proposition 4.1]{LapMin16}}]
    Let $b\in \NN,\,\rho\in \cus,\, \fm\in \Ms$. Then there exists $ a \in \NN,a\le b$ such that $\Z(\fm)$ is not $[a,b]_\rho$-reduced if and only if there exists a segment in $\fm^*$ of the form $[c,b]_\rho$. 
    If there exists such $a$, let it be minimal. Then the number of segments in $\fm^*$ of the form $[c,b]_\rho$ equals $l_{[a,b]_\rho}(\Z(\fm))$.
\end{lemma}

Let us finally state an explicit version for the Geometric Lemma in the case of the general linear group.
Let $\ain{r,t}{1}{n}$ and let \[\mathfrak{A}(t,r)=\{(k_1,k_2)\in\NN^2: k_1+k_2=t, \, 0\le r-k_1\le n-t\}.\]
For $(k_1,k_2)\in\mathfrak{A}(t,r)$ one can find $w_{k_1,k_2}\in \Gl'_n$ such that conjugation by $w_{k_1,k_2}$ acts on the Levi-subgroup of $\bP_{k_1,k_2,r-k_1,n-t-r+k_1}$ by acting trivially on $\Gl'_{k_1}$ and $\Gl'_{n-t-r+k_1}$ and swapping $\Gl'_{k_2}$ and $\Gl'_{r-k_1}$.
\begin{lemma}[Geometric Lemma for $\Gl'_n$]
    Let $\ain{r,t}{1}{n}$, $\pi_1\otimes\pi_2\in\Rep(\gl'_t\times\gl'_{n-t})$. Then $r_{r,n-r}(\pi_1\times\pi_2)$ has a filtration whose subquotients are indexed by $(k_1,k_2)\in\mathfrak{A}(t,r)$ and which are of the following form:
    \[\Pi_{k_1,k_2}=\id_{P_{k_1,r-k_1,k_2,n-t-r+k_1}}^{\gl'_r\times\gl'_{n-r}}(\mathrm{Ad}(w_{k_1,k_2})(r_{k_1,k_2}(\pi_1)\otimes r_{r-k_1,n-t-r+k_1}(\pi_2))).\]
\end{lemma}
\subsection{The symplectic group}\label{S:symplecticgroup}
Whenever we deal with symplectic groups or their metapletic covers, we implicitly assume that $\mathrm{char}\,\Ff\neq 2$.
We equip $\AA^{2n}$ with the standard symplectic form \[\begin{pmatrix}
    0&1_n\\-1_n&0
\end{pmatrix}\] and consider its symmetry group, the symplectic group $\Sp_{n}$. For $\ain{i}{1}{n}$ we let $\AA^i_\pm$ be the subspace of $\AA^{2n}$ given by the first respectively last $i$-coordinates. For $\alpha=(\alpha_1,\ldots,\alpha_k)$ a partition of $m\le n$, we denote the parabolic subgroup given by the stabilizer of the flag
\[\{0\subseteq\AA_+^{\alpha_1}\subseteq \ldots\subseteq \AA_+^{m}\subseteq \AA^{2n}\}\]by $\bP_\alpha$. The Levi-component of $\bP_\alpha$ is isomorphic to \[\Gl_{\alpha_1}\times\ldots\times\Gl_{\alpha_k}\times \Sp_{n-m}.\]
Recall that $\overline{\bP_\alpha}$ is conjugated to $\bP_\alpha$ by an element which acts on $\Gl_{\alpha_i}$, $\ain{i}{1}{k}$ by $g\mapsto {}^tg^{-1}$. Moreover, we have that\[\delta_{P_\alpha}=\nu_{\alpha_1} ^{p(\alpha)_1}\otimes\ldots\otimes \nu_{\alpha_t} ^{p(\alpha)_t},\, p(\alpha)_i\coloneq n- (2\sum_{j=1}^{i-1}\alpha_j)-\alpha_i+1.\]

We then write for \[\pi_1\otimes\ldots\otimes \pi_k\otimes \tau\in \Rep_{\alpha_1}\times \ldots \times \Rep_{\alpha_k}\times \Rep(\sp_{m-n})\] the parabolically induced representation as
\[\pi_1\times\ldots\times \pi_k\rtimes\tau\coloneq \id_{P_\alpha}^{\sp_n}(\pi_1\otimes\ldots\otimes \pi_k\otimes \tau).\]
 Recall the M{\oe}glin-Vign\'eras-Waldspurger involution, \emph{cf.} \cite[p.91]{MVW}, \[\mathrm{MVW}\colon \Rep(\sp_n)\rightarrow \Rep( \sp_n)\] which is covariant, exact and satisfies
\begin{enumerate}
    \item $\mathrm{MVW}\circ\mathrm{MVW}=\mathrm{id}$,
    \item $\pi\mvw\cong \pi^\lor$ if $\pi$ is irreducible,
    \item For $\ain{m}{1}{n}$, $\alpha=(\alpha_1,\ldots,\alpha_k)$ a partition of $m$, $\pi_i\in \irr_{\alpha_i}$, $\tau\in\irr(\sp_{n-m})$,\[(\pi_1\times\ldots\times \pi_k\rtimes \tau)\mvw\cong\pi_1\times\ldots\times \pi_k\rtimes\tau\mvw.\]
\end{enumerate}
Let $\ain{t,r}{1}{n}$ and
set \[\mathfrak{B}(t,r)\coloneq\{(k_1,k_2,k_3)\in \NN^3:\, k_1+k_2+k_3=t,\, k_1+k_3\le r,\, r+k_2\le n\}.\]
 To $(k_1,k_2,k_3)\in \mathfrak{B}(t,r)$ we can associate a certain element $w_{k_1,k_2,k_3}\in \Gl_{2n}$ with the following properties. Conjugating the Levi-subgroup of $\bP_{k_1,k_2,k_3,r-k_1-k_2}$ by $w_{k_1,k_2,k_3}$ exchanges the factors of $\Gl_{k_2}$ and $\Gl_{r-k_1-k_2}$, but leaves them otherwise unchanged, and acts on $\Gl_{k_1}$ and $\Sp_{n-r-k_3}$ trivially. Finally, $w_{k_1,k_2,k_3}$ preserves $\Gl_{k_3}$. If $\pi_3$ is an irreducible representation of $\Gl_{k_3}$, twisting the action by conjugation by $w_{k_1,k_2,k_3}$ yields the representation $\pi_3^\lor$.
 
\begin{lemma}[{Geometric Lemma for $\sp_n$, \cite{Ber}, \cite[Lemma 5.1]{TADIC19951}}]\label{L:GL}
Let $\ain{t,r}{1}{n}$ and $\pi\otimes\tau\in\Rep(\gl_t\times\sp_{n-t})$.
Then $r_{P_r}(\pi\rtimes\tau)$ admits a filtration whose subquotients are indexed by $(k_1,k_2,k_3)\in \mathfrak{B}(t,r)$ and are of the following form: \[\Pi_{k_1,k_2,k_3}=\id_{P_{k_1,r-k_1,k_2}\times P_{k_2}}^{\gl_r\times \sp_{n-r}}(\mathrm{Ad}(w_{k_1,k_2,k_3})(r_{k_1,k_2,k_3}(\pi)\otimes r_{P_{r-k_1-k_3}}(\tau))).\]
\end{lemma}
\subsection{Metaplectic covers}
As is common in the study of phenomena related to the theta correspondence, one is often forced to consider certain double covers of general and symplectic groups which depend on the choice of an additive character $\psi\colon\Ff\ra\C^\times$
\subsubsection{The group $\Glt_n$}\label{S:dergl}\label{S:2gl}
We set $\Glt_n$ to the twofold cover of $\gl_n$ which is 
\[\gl_n\times \mu_2\]
with multiplication \[(g,\epsilon)\cdot (g',\epsilon')=(gg',\epsilon\epsilon'\cdot(\det(g),\det(g'))_\Ff),\]
where $(-,-)_\Ff$ is the Hilbert symbol of $\Ff$.
For $\alpha$ a partition of $n$, we let $\tp_\alpha$ be the preimage of the parabolic subgroup $P_\alpha$ of $\gl_n$. We write the Levi-decomposition of $\tp_\alpha$ as $\widetilde{L}_\alpha\widetilde{N}_\alpha$, where $\widetilde{L}_\alpha=\Glt_{\alpha_1}\times^{\mu_2}\ldots\times^{\mu_2}\Glt_{\alpha_k}$ and $\widetilde{N}_\alpha=N_\alpha$, since the covering maps splits over unipotent subgroups. 

There exists a bijection \[(-)^\psi\colon \Rep_n\ra\Rep(\Glt_n)^{gen},\] which depends on our chosen additive character $\psi$ of $\Ff$ and is constructed as follows. Here we use the notation $\Rep(\Glt_n)^{gen}$ for the set of \emph{genuine} representations of $\Glt_n$, \emph{i.e.}, the ones where $\mu_2$ does not act trivially. It is clear that the representations on whom $\mu_2$ acts trivially are obtained by pullback from $\gl_n$. 
First lift $\det_n\colon \gl_n\rightarrow \gl_1$ to \[\widetilde{\det_n}\colon \Glt_n\ra\Glt_1\] by $(g,\epsilon)\mapsto (\det(g),\epsilon)$.
Let $\gamma_\Ff(\psi)\in \mu_8$ be the Weil index of $\psi$.
For $a\in \Ff$, let $\psi_a\coloneq \psi(a\,\cdot)$ and set
\[\gamma_\Ff(a,\psi)\coloneq \frac{\gamma_\Ff(\psi_a)}{\gamma_\Ff(\psi)}.\]
This gives a character \[\chi_{-1}^\frac{1}{2}\colon\Glt_1\ra \C^\times,\, (a,\epsilon)\mapsto \epsilon\gamma_\Ff(a,\psi)^{-1}.\]
Composing with $\widetilde{\det}$ gives a character $\chi_\psi$ of $\Glt_n$.
Finally, we define the bijection $\Rep_n\ra\Rep(\Glt_n)^{gen}$ by sending \[\pi\mapsto \pi^\psi\coloneq \pi\otimes\chi_\psi.\]
The results of \Cref{S:repgl} now translate straightforwardly to the case $\Glt_n$. For a more detailed discuss see for example \cite[§7]{KLZ} .
\subsubsection{The group $\mp_n$}
We will also encounter the metaplectic group $\mp_n$, which sits in the short exact sequence
\[0\ra \mu_2\ra\mp_n\ra\sp_n\ra0.\]
For a partition $\alpha$ of $m\le n$, we let $\tp_\alpha$ be the preimage of the parabolic subgroup $P_\alpha$ of $\sp_n$. We write the Levi-decomposition of $\tp_\alpha$ as $\widetilde{L}_\alpha\widetilde{N}_\alpha$, where $\widetilde{L}_\alpha=\Glt_{\alpha_1}\times^{\mu_2}\ldots\times^{\mu_2}\Glt_{\alpha_k}\times^{\mu_2}\mp_{n-m}$ and $\widetilde{N}_\alpha=N_\alpha$, since again, the covering maps splits over unipotent subgroups. We thus are able to extend the notions of \Cref{S:symplecticgroup} straightforwardly to the metaplectic case.
\section{Representation theoretical results}
In this section we introduce some representation-theoretical notions and results for above groups, which will be important in the later parts of the paper.
\subsection{Intertwining operators}
In this section we recall intertwining operators and several of their properties. Let $\bG$ be a reductive $\Ff$-group and let $\bP=\bL\bN\subseteq \bG$ be a parabolic subgroup with opposite parabolic subgroup $\overline{\bP}$. Following \cite[§IV]{Wal03}, we denote for $\pi\in\Rep(L)$ the normalized intertwining operator \[I_{P,\pi}=I_\pi\colon\id_{\op}^G(\pi)\ra \id_P^G(\pi)\]
and the order of its pole by $\Lambda(\pi)\in \ZZ_{\ge 0}$. We also write $\Lambda(\pi_1,\ldots,\pi_k)=\Lambda(\pi_1\otimes\ldots\otimes\pi_k)$.
If $\bG=\Gl_n'$, $\bP=\bP_\alpha'$ and $\pi=\pi_1\otimes\ldots\pi_k\in \Rep(L_\alpha')$ for a partition $\alpha=(\alpha_1,\ldots,\alpha_k)$ of $n$, we recall that $\overline{\bP_\alpha'}$ is conjugated to $\bP_{\overline{\alpha}}'$ by a suitable element $w$ and denote the so obtained morphism by 
\[I_{\pi_1,\ldots,\pi_k}=I_{\pi_1\otimes\ldots\otimes\pi_k}\circ\mathrm{Ad}(w)\colon \pi_k\times\ldots\times \pi_1\ra\pi_1\times\ldots\times \pi_k.\]
If $\bG=\Sp_n$ with $\bP=\bP_\alpha$ and $\pi=\pi_1\otimes\ldots\pi_k\otimes \in\Rep(L_\alpha)$ for a partition $\alpha=(\alpha_1,\ldots,\alpha_k)$ of $m\le n$
we note that $\overline{\bP_\alpha}$ is conjugated to $\bP_{\alpha}$ by a suitable element $v$ and denote the so obtained morphism by 
\[I_{\pi_1,\ldots,\pi_k,\tau}=I_{\pi_1\otimes\ldots\otimes\pi_k\otimes\tau}\circ\mathrm{Ad}(v)\colon \pi_1\cc\times\ldots\pi_k\cc\rtimes\tau\ra\pi_1\times\ldots\times\pi_k\rtimes\tau.\]
We can extend these notations straightforwardly to the respective metaplectic covers. 
Let us remark the following. Assume that $\bP$ is conjugated to itself by an element $w$ and $\sigma$ a representation of $\bL$ such that $\mathrm{Ad}(w)(\sigma)=\sigma$. Then \[I_{P,\sigma}\circ\mathrm{Ad}(w)= \mathrm{Ad}(w)\circ I_{\overline{P},\sigma},\]
see \cite[p. 286]{Wal03}.
\begin{lemma}\label{L:interop}
    Let $G\in \{\gl_n,\gl_n',\Glt_n,\sp_n,\mp_n\}$ and $\pi$ as above. Then the following holds.
    \begin{enumerate}
        \item Let $0\ra\pi'\ra\pi\ra\pi''\ra0$ be a short-exact sequence. Then $\restr{I_\pi}{\pi'}=I_{\pi'}$ if $\Lambda(\pi)=\Lambda((\pi')$ and $0$ otherwise. Thus, $I_\pi$ induces a natural morphism $\id_{\op}^G(\pi'')\ra \id_P^G(\pi'')$. This morphism is equal to $I_{\pi''}$ if $\Lambda(\pi)=\Lambda((\pi'')$ and $0$ otherwise.
        \item The composition $I_{P,\pi}\circ I_{\overline{P},\pi}$ is always a scalar. In particular, if it is non-zero, it is an isomorphism.
        \item Assume $G\in\{\sp_n,\mp_n\}$ and $\pi=\pi_1\otimes\pi_2\otimes\tau$. Then $\Lambda(\pi_1,\pi_2\rtimes\tau)+\Lambda(\pi_2,\tau)=\Lambda(\pi_2,\pi_2)+\Lambda(\pi_1\times\pi_2,\tau)\ge \Lambda(\pi_1,\pi_2,\tau)$.
        Moreover, the diagram
    \[\begin{tikzcd}
        &\pi_1\rtimes(\pi_2\cc\rtimes\tau)\arrow[dr,"1_{\pi_1}\rtimes I_{\pi_2,\tau}"]&\\
        \pi_1\cc\times\pi_2\cc\rtimes\tau\arrow[rr, "I_{\pi_1,\pi_2,\tau}"]\arrow[ur,"I_{\pi_1,\pi_2\cc\rtimes\tau}"]\arrow[dr,"(I_{\pi_1\cc,\pi_2\cc})\rtimes 1_\tau"']&&\pi_1\times\pi_2\rtimes\tau\\
        &\pi_2\cc\times\pi_1\cc\rtimes\tau\arrow[ur,"I_{\pi_1\times\pi_2,\tau}"']&
    \end{tikzcd}\]
    commutes, where the middle arrow is $0$ if above inequality is strict and $I_{\pi_1,\pi_2,\tau}$ if above inequality is an inequality.
    \end{enumerate}
\end{lemma}
In particular, if $\pi_1\in\irr_n$ is self dual and $\pi_2\in\irr(\sp_n)$, we obtain a non-zero map
\[I_{\pi_1,\pi_2}\colon \pi_1\rtimes\pi_2\ra\pi_1\rtimes\pi_2\]
such that $I_{\pi_1,\pi_2}\circ I_{\pi_1,\pi_2}$ is a, possible vanishing, scalar.
\begin{theorem}[{\cite[§2]{MinLa18}}]\label{T:si}
A representation $\pi\in \Irr_n$ is called $\square$-irreducible if one of the following equivalent conditions holds.
\begin{enumerate}
    \item $\pi\times\pi\in \Irr$.
    \item $I_{\pi,\pi}$ is a scalar.
    \item For all $\sigma\in \Irr$, $\pi\times\sigma$ is SI.
    \item For all $\sigma\in \Irr$, $\sigma\times \pi$ is SI.
    \item For all $\sigma\in \Irr$, \[\dim_\C\ho_{G_{\deg(\pi\times\sigma)}}(\pi\times\sigma,\sigma\times\pi)=\dim_\C\ho_{G_{\deg(\pi\times\sigma)}}(\sigma\times\pi,\pi\times\sigma)=1.\]
\end{enumerate}
\end{theorem}
We denote the subset of $\Irr$ consisting of $\square$-irreducible representations by $\irs$. 
\begin{lemma}[{\cite[Theorem 4.1.D]{LapMin25}}]
    Let $\pi\in\irs$. The two maps
    \[\soc(\pi\times-)=\cosoc(-\times\pi),\soc(-\times\pi)=\cosoc(\pi\times-)\colon\Irr\ra\Irr\]
    are injective.
\end{lemma}
\begin{lemma}[{\cite[Lemma 2.10]{MinLa18}}]\label{L:prodsquare}
    Let $\pi,\pi'\in\irs$ such that $\pi\times\pi'\in\Irr$. Then $\pi\times\pi'\in\irs$.
\end{lemma}
For example all irreducible representations of the form of \Cref{T:zel} are $\square$-irreducible. The following was proven in the proof of \cite[Theorem 4.1.D]{LapMin25}.
\begin{lemma}\label{L:lapminint}
    Let $\rho$ be $\square$-irreducible representation of $\gl_m'$ and $\pi\in \irr_n'$, $n\ge m$. Let $\tau$ be the maximal $(\rho,\gl_m')$-equivariant subrepresentation of $\overline{r}_{n-m,m}(\pi)$, \emph{i.e.} the maximal subrepresentation of the form $\tau\otimes \rho$. Then $I_{\rho,\tau}$ factors through the map $\tau\times\rho\sra \pi$ obtained by Bernstein-reciprocity.
\end{lemma}
An analogous version of the lemma holds true if one considers the $(\rho,\gl_m)$-equivariant quotient of $r_{m,n-m}(\pi)$. Finally, let us note that the proof of the result can be extended by the same argument to the case $\Glt_n$. 
\subsection{$\rho$-derivatives}
In this section we recall derivatives of classical and general groups and following the treatment of \cite[Appendix C.1]{atobe2024localintertwiningrelationscotempered}.
\subsubsection{The case of $\Gl_n$}
In here we explicate some of the things mentioned in \Cref{S:deltader} in the case where $\De$ is a segment of length $1.$

Let $\rho\in\cus_m$. We call a representation $\pi\in \Rep_n$ right- (respectively left-$\rho$-reduced) if $n\ge m$ and there exists no $\tau\in\irr_{n-m}$ such that $\rho\otimes\tau$ appears as a subquotient of $r_{m,n-m}(\pi)$ (respectively $\tau\otimes\rho$ appears as a subquotient of $r_{n-m,m}(\pi)$).
\begin{lemma}
    Let $\pi\in \irr_n$ and $\rho\in\cus_m$. Then there exists $\tau\in \irr_{n-m}$ with $\pi\hra\tau\times\rho$ if and only if $\pi$ is not right-$\rho$-reduced. If $\pi$ is not right-$\rho$-reduced, $\tau$ is unique and denoted by $\dr(\pi)$. 
    Moreover, in this case $\pi$ is the unique subrepresentation of $\dr(\pi)\times\rho$, the unique quotient of $\rho\times \dr(\pi)$ and appears in the respective composition series with multiplicity one.
    If $\pi$ is right-$\rho$-reduced, we set $\dr(\pi)\coloneq 0$. 
\end{lemma}
We set $\drm(\pi)$ to be the maximal $k\in\NN$ such that \[\dr^k(\pi)\coloneq \overbrace{\dr\circ\ldots\circ\dr}^{k}(\pi)\neq 0\] and write $\Dr(\pi)\coloneq \dr^{\drm(\pi)}(\pi)$.
An analogous lemma holds for $\dl$ and allows us to define $\dl(\pi)$, $\dlm(\pi)$ and $\Dl(\pi)$.
\begin{lemma}
    Let $\pi\in \irr_n$ and $\rho\in\cus_m$. Then $\pi$ is the unique subrepresentation of $\Dr(\pi)\times\rho^{\times \drm(\pi)}$, the unique quotient of $\rho^{\times \drm(\pi)}\times \Dr(\pi)$ and appears in the respective composition series with multiplicity one. Moreover, for all other irreducible subquotients $\sigma$ of $\Dr(\pi)\times\rho^{\times \drm(\pi)}$, $\drm(\sigma)<\drm(\pi)$.

    Finally, $\Dr(\pi)^\lor=\mathcal{D}_{\rho^\lor,l}(\pi^\lor)$.
\end{lemma}
We will state the following lemmas only for right-derivatives. It is clear that analogous statements hold for left-derivatives. 
\begin{lemma}\label{L:derivativesintgl}
    Let $\pi\in \Rep_n$ of finite length, $\rho\in\cus_m$ and $k\in \NN$. If $\pi$ is right-$\rho$-reduced, $\Lambda(\rho^{\times k},\pi)=0$ and if $\pi'$ is a subrepresentation or quotient of $\pi$, $I_{\rho^{\times k},\pi}$ induces the map $I_{\rho^{\times k},\pi'}\colon\rho^{\times k}\times\pi'\ra\pi'\times\rho^{\times k}$. If $\pi$ is irreducible, the map $I_{\rho^{\times k},\pi}$ factors through $\soc(\pi\times\rho^{\times k})$.
    In particular 
    \[\mathrm{cosoc}(\mathrm{Im}(I_{\rho^{\times k},\pi}))=\soc(\cosoc(\pi)\times \rho^{\times k}).\]
\end{lemma}
\begin{proof}
   Note since $\pi$ is right-$\rho$-reduced, one can deduce by \cite[Theorem IV.1.1.]{Wal03}, that for $f\in \rho^{\times k}\times\pi$ whose support is contained in $P_{km,n}wP_{n,km}$, where $w$ is the anti-diagonal, 
   \[I_{\rho^{\times k},\pi}(f)(1)=\int_{N_{n,km}}f(n)\mathrm{d}n\] for a suitable Haar-measure. This implies by \cite[Theorem IV.1.1.]{Wal03} that $\Lambda(\rho^{\times k},\pi)=0$. The second claim is now an easy consequence of the first by \Cref{L:interop}.
    For the final claim it is enough to note that there exists at most one, up to a scalar, non-zero morphism
    \[\rho^{\times k}\times\pi\ra\pi\times \rho^{\times k},\]
    as any easy consequence of Frobenius reciprocity and the Geometric Lemma. Then the claim follows by the above mentioned properties of derivatives.
\end{proof}
\begin{lemma}\label{L:derrek}
    Let $\pi\in \Rep_n,\pi'\in \Rep_r$ and $\rho\in \cus$ such that $\pi\times \rho^{\times k}$ is irreducible for all $k\in \NN$ and $\pi'$ is left-$\rho$-reduced. Then for any irreducible subrepresentation $\sigma$ of $\pi\times \rho^{\times k}\times\pi'$ the following are equivalent
       \begin{enumerate}
        \item  $\sigma\hra\pi\times\soc(\rho^k\times \pi')).$
        \item $\sigma\hra\rho^{\times \dm(\pi)+k}\times \soc(\Dd(\pi)\times \pi').$
        \item $\sigma\hra\rho^{\times \dm(\pi)+k}\times \Dd(\pi)\times \pi'.$ 
    \end{enumerate}
    In particular, if $\Dd(\pi)\times \pi'$ has an irreducible socle, the so does $\pi\times\soc(\rho^k\times \pi'))$.
\end{lemma}
\begin{proof}
The implication $(ii)\Rightarrow (iii)$ is trivially true.

Assume next that \[\sigma\hra\pi\times\soc(\rho^k\times \pi')).\]Therefore
 \[\sigma\hra \pi\times\rho^k\times\pi'\cong \rho^k\times\pi\times\pi'\hra \rho^{\times \dlm(\pi)+k}\times \Dl(\pi)\times \pi'.\]
Hence also $(i)\Rightarrow (iii)$ follows.
Let $\sigma'$ be an irreducible representation such that
\[\sigma\hra \rho^{\times \dlm(\pi)+k}\times \sigma'\]
and hence
\[\sigma'\times\rho^{\times \dlm(\pi)+k}\sra \pi\hra \rho^{\times \dlm(\pi)+k}\times \Dl(\pi)\times \pi'.\]
Bernstein reciprocity and the Geometric Lemma thus imply that $\sigma'\hra \Dl(\pi)\times \pi'$ and hence $(i)\Rightarrow (ii)$ follows.
Finally, assume \[\sigma\hra\rho^{\times \dlm(\pi)+k}\times \Dl(\pi)\times \pi'\]
and in particular $\dlm(\sigma)=k+\dlm(\pi)$.
Let $\sigma_1$ be such that $\sigma_1$ is an irreducible subquotient of 
\[\rho^{\times \dlm(\pi)}\times \Dl(\pi)\] and 
\[\sigma\hra \
\rho^{\times k}\times\sigma_1\times \Dl(\pi).\]
Since $\dlm(\sigma)=k+\dlm(\pi)$, it follows that $\dlm(\sigma_1)=\dlm(\pi)$ and hence $\sigma_1=\pi$. Thus
\[\sigma\hra\pi\times\rho^{\times k}\times \pi'.\] The same argument gives now that 
\[\sigma\hra\rho^{\times \dlm(\pi)+k}\times \soc(\Dl(\pi)\times \pi'),\] proving the direction $(iii)\Rightarrow(i)$.
\end{proof}

Let us end this section by noting that all results here have analogous statements when one replaces the group $\gl_n$ by $\Glt_n$ or $\gl_n'$.
\subsubsection{The case of $\sp_n$}\label{S:dersp}
The theory of derivatives has been extended in \cite{AtoMin} to the case of classical groups and we will now recall its most important aspects.

Let $\rho\in\cus_m\setminus\cuu_m$. We call a representation $\pi\in \Rep(\sp_n)$ $\rho$-reduced if $n\ge m$ and there exists no $\tau\in\irr_{n-m}$ such that $\rho\otimes\tau$ appears as a subquotient of $r_{P_m}(\pi)$.
\begin{lemma}
    Let $\pi\in \irr(\sp_n)$ and $\rho\in\cus_m\setminus\cuu_m$. Then there exists $\tau\in \irr(\sp_{n-m})$ with $\pi\hra\rho\rtimes\tau$ if and only if $\pi$ is not $\rho$-reduced. If $\pi$ is not $\rho$-reduced, $\tau$ is unique and denoted by $\dd(\pi)$. 
    Moreover, in this case $\pi$ is the unique subrepresentation of $\rho\rtimes\dd(\pi)$, the unique quotient of $\rho^\lor\rtimes \dd(\pi)$, and appears in the respective composition series with multiplicity one.
    If $\pi$ is $\rho$-reduced, we set $\dd(\pi)\coloneq 0$. 
\end{lemma}
We set $\dm(\pi)$ to be the maximal $k\in\NN$ such that \[\dd^k(\pi)\coloneq \overbrace{\dd\circ\ldots\circ\dd}^{k}(\pi)\neq 0\] and write $\Dd(\pi)\coloneq \dd^{\dm(\pi)}(\pi)$.
\begin{lemma}
    Let $\pi\in \irr(\sp_n)$ and $\rho\in\cus_m\setminus\cuu_m$. Then $\pi$ is the unique subrepresentation of $\rho^{\times \dm(\pi)}\times \Dd(\pi)$, the unique quotient of $(\rho^{\times \dm(\pi)})^\lor\times \Dd(\pi)$ and appears in the respective composition series with multiplicity one.

    Finally, $\Dd(\pi)^\lor=\Dd(\pi^\lor)$.
\end{lemma}
The following can be proven analogously to \Cref{L:derivativesintgl}.
\begin{lemma}\label{L:derivativesint}
    Let $\pi\in \Rep(\sp_n)$ of finite length, let $\rho\in\cus_m\setminus\cuu_m$ and $k\in \NN$. If $\pi$ is $\rho$-reduced, $\Lambda(\rho^{\times k},\pi)=0$ and if $\pi'$ is a subrepresentation or quotient of $\pi$, $I_{\rho^{\times k},\pi}$ restricts to $I_{\rho^{\times k},\pi'}$. If $\pi$ is irreducible, the map $I_{\rho^{\times k},\pi}$ factors through $\soc(\pi\times\rho^{\times k})$.
        In particular 
    \[\mathrm{cosoc}(\mathrm{Im}(I_{\rho^{\times k},\pi}))=\soc(\cosoc(\pi)\times \rho^{\times k}).\]
\end{lemma}
The following can be proved analogously to \Cref{L:derrek}.
\begin{lemma}\label{L:derreksp}
    Let $\pi\in \Rep_n,\pi'\in \Rep(\sp_r)$ and $\rho\in \cus_m\setminus\cuu_m$ such that $\pi\times \rho^{\times k}$ is irreducible for all $k\in \NN$, $\pi$ is right-$\rho^\lor$-reduced, and $\pi'$ is left-$\rho$-reduced.
    Then for any irreducible subrepresentation $\sigma$ of $\pi\times\rho^{\times k}\rtimes\pi'$ the following are equivalent:
    \begin{enumerate}
        \item  $\sigma\hra\pi\rtimes\soc(\rho^k\rtimes \pi')).$
        \item $\sigma\hra\rho^{\times \dm(\pi)+k}\rtimes \soc(\Dd(\pi)\rtimes \pi').$
        \item $\sigma\hra\rho^{\times \dm(\pi)+k}\times \Dd(\pi)\rtimes \pi'.$ 
    \end{enumerate}
        In particular, if $\Dd(\pi)\rtimes \pi'$ has an irreducible socle, the so does $\pi\times\soc(\rho^k\rtimes \pi'))$.
\end{lemma}
Let us end this section by noting that all results here have analogous statements when one replaces the group $\sp_n$ by $\mp_n$.
\subsection{$\De$-derivatives for symplectic groups}\label{S:Deder}
The goal of this section is to prove the following theorem.
\begin{theorem}
    Let $\De=[a,b]_\rho,\, \rho\in\cuu_d$ be a segment of degree $m$ and $\pi\in \irr(\sp_n)$. Then $\soc(\Z(\De)\rtimes\pi)$ is multiplicity-free. If $\tau\hra \Z(\De)\rtimes\pi$ is an irreducible subrepresentation and $\pi'$ is a second irreducible representation such that $\tau\hra \Z(\De)\rtimes\pi'$, we have $\pi\cong \pi'$. 
\end{theorem}
This theorem will allow us to define the map
\[\mathcal{D}_\De\colon \irr(\sp_{n+m})\ra \irr(\sp_n),\]
sending $\tau$ to the irreducible representation $\pi$ with $\tau\hra \Z(\De)\rtimes\pi$, if it exists, and $0$ otherwise.
We start with the following corollary of \Cref{L:lapminint}, see also \cite[Lemma 2.9.3]{DroKudRa}.
Throughout the section we will only prove the results in the case of $\sp_n$, however it is just a matter of adjusting the notation to carry out the same proofs for its metaplectic cover.
\begin{lemma}\label{L:ext}
    Let $a\le b\in \NN$ and $\ain{a=s_1\le\ldots\le s_k}{a}{b}$.
    Let $\rho\in\cus$ and set $\De_i\coloneq [s_i,b]_\rho$ for $\ain{i}{1}{k}$.
    Denote by $m=\deg(\De_1)$ and let $n\coloneq \deg(\De_2)+\ldots+\deg(\De_k)$. Let $\Pi\in \Rep_{n}$ such that $\Z(\De_1)\otimes \Pi\hra\overline{r}_{(m,n)}(\Z(\De_1)\times\ldots\times\Z(\De_k))$. Then $\Pi\cong \Z(\De_2)\times\ldots\times \Z(\De_k)$.
\end{lemma}
\begin{proof}
    By \Cref{L:lapminint} we have that \[\Z(\De_1)\times\ldots\times \Z(\De_k)=\im(I_{\Z(\De_1),\Pi}).\] 
As a consequence of the Geometric Lemma, we have that all irreducible subquotients of $\Pi$ are isomorphic to $\Z(\De_2)\times\ldots\times \Z(\De_k)$.
But \Cref{L:interop} (1) shows then that for any short exact sequence $0\ra\Pi_1\ra\Pi\ra\Pi_2\ra 0$, we have the following commutative diagram
    \[\begin{tikzcd}
        0\arrow[r]&\Z(\De_1)\times\Pi_1\arrow[r]\arrow[d,"I_{\Z(\De_1),\Pi_1}"]&\Z(\De_1)\times\Pi\arrow[r]\arrow[d,"I_{\Z(\De_1),\Pi}"]&\Z(\De_1)\times\Pi_2\arrow[r]\arrow[d,"I_{\Z(\De_1),\Pi_2}"]&0\\
        0\arrow[r]&\Pi_1\times\Z(\De_1)\arrow[r]&\Pi\times \Z(\De_1)\arrow[r]&\Pi_2\times\Z(\De_1)\arrow[r]&0.
    \end{tikzcd}\]
    
    This shows that in this case the image of $I_{\Z(\De_1),\Pi}$ has to have at least the same length as $\Pi$, which implies that $\Pi$ is irreducible.
\end{proof}
We call a segment $\De=[a,b]_\rho$ self-dual if $\rho \nu_\rho^a\cong \rho^\lor \nu_\rho^{-b}$, or equivalently, if $\Z(\De)^\lor\cong \Z(\De)$.
\begin{lemma}\label{L:2dim}
     Let $\pi\in \irr(\sp_n)$ and $\De$ a self dual segment of degree $m$. Then
    \[\dim_\C\ho_{\sp_{m+n}}(\Z(\De)\rtimes\pi,\Z(\De)\rtimes \pi)\le 2.\]
\end{lemma}
\begin{proof}
    Write $\De=[a,b]_\rho$ and write $\pi\hra \Z([s_1,b]_\rho)\times \ldots \times \Z([s_k,b]_\rho)\rtimes \pi'$ with $a\le s_1\le s_2\le\ldots\le b$ and there exists no $\ain{s'}{a}{b}$ and suitable $\pi''$ with $\pi\hra\Z([s',b]_\rho)\rtimes\pi''$. We write $r=\deg(\Z([s_1,b]_\rho))+ \ldots + \deg(\Z([s_k,b]_\rho))$.
    We have on the one hand 
    \[\dim_\C\ho_{\sp_n}(\pi,\Z([s_1,b]_\rho)\times \ldots \times \Z([s_k,b]_\rho)\rtimes \pi')\cdot \dim_\C\ho_{\sp_{m+n}}(\Z(\De)\rtimes\pi,\Z(\De)\rtimes \pi)\le\]\[\le \dim_\C\ho_{\sp_{m+n}}(\Z(\De)\rtimes\pi,\Z(\De)\times \Z([s_1,b]_\rho)\times \ldots \times \Z([s_k,b]_\rho)\rtimes \pi').\]
    On the other hand, Bernstein reciprocity, the Geometric Lemma and the assumption on $\pi''$ imply that 
    \[ \dim_\C\ho_{\sp_{m+n}}(\Z(\De)\rtimes\pi,\Z(\De)\times \Z([s_1,b]_\rho)\times \ldots \times \Z([s_k,b]_\rho)\rtimes \pi')\le\]
    \[2\dim_\C\ho_{\gl_m\times \sp_{n}}(\Z(\De)\otimes\pi,\]\[ \overline{r}_{m,r}(\Z(\De)\times \Z([s_1,b]_\rho)\times \ldots \times \Z([s_k,b]_\rho))\rtimes \pi').\]
    By \Cref{L:ext} the above is equal to 
      \[2\dim_\C\ho_{\gl_m\times \sp_{n}}(\pi,\Z([s_1,b]_\rho)\times \ldots \times  \Z([s_k,b]_\rho)\rtimes \pi').\]
We thus obtain from the first inequality that \[\dim_\C\ho_{\sp_{m+n}}(\Z(\De)\rtimes\pi,\Z(\De)\rtimes \pi)\le 2.\] 
\end{proof}
\begin{lemma}\label{L:2dimstronger}
    Let $\pi\in \irr(\sp_n)$ and $\De$ a self dual segment of degree $m$. Then
    \[\dim_\C\ho_{\sp_{m+n}}(\Z(\De)\rtimes\pi,\Z(\De)\rtimes \pi)=2\] unless $\Z(\De)\rtimes \pi$ is irreducible. If $\soc(\Z(\De)\rtimes \pi)$ is not irreducible, $\Z(\De)\rtimes\pi$ is semi-simple of length $2$ and the two summands are non-isomorphic.
\end{lemma}
\begin{proof}
Let $c_{\tau}$ be the multiplicity of an irreducible representation $\tau$ in the socle of $\Z(\De)\rtimes\pi$.
The MVW-involution gives the lower bound
\[\dim_\C\ho_{\sp_{m+n}}(\Z(\De)\rtimes\pi,\Z(\De)\rtimes \pi)\ge \sum_\tau c_\tau^2.\]
Moreover, if $\soc(\Z(\De)\rtimes\pi)\neq \Z(\De)\rtimes\pi$, the identity $\Z(\De)\rtimes\pi\ra \Z(\De)\rtimes \pi$ is not accounted for in the above inequality, hence if  $\Z(\De)\rtimes\pi$ is not semi-simple
\[\dim_\C\ho_{\sp_{m+n}}(\Z(\De)\rtimes\pi,\Z(\De)\rtimes \pi)\ge \sum_\tau c_\tau^2+1.\]
The claim then follows from \Cref{L:2dim}.
\end{proof}
\begin{lemma}\label{L:deruniqueseg}
     Let $\pi\in \irr(\sp_n)$ and $\De$ a self dual segment of degree $m$. Let $\tau$ be an irreducible subrepresentation of $\Z(\De)\rtimes \pi$. If $\sigma\in \irr(\sp_n)$ with $\tau\hra \Z(\De)\rtimes \sigma$ then $\pi\cong \sigma$. 
\end{lemma}
\begin{proof}
    Write $\De=[a,b]_\rho$ and write $\pi\hra \Z([s_1,b]_\rho)\times \ldots \times \Z([s_k,b]_\rho)\rtimes \pi'$ with $a\le s_1\le s_2\le\ldots\le s_k$ and there exists no $\ain{s'}{a}{b}$ and suitable $\pi''$ with $\pi\hra\Z([s',b]_\rho)\rtimes\pi''$. As in \Cref{L:2dim} it follows that $\sigma\hra \Z([s_1,b]_\rho)\times \ldots \times \Z([s_k,b]_\rho)\rtimes \pi'$. Assuming $\sigma\ncong \pi$, it follows, again by an argument as in \Cref{L:2dim}, that  \[\dim_\C\ho_{\sp_n}(\pi,\Z([s_1,b]_\rho)\times \ldots \times \Z([s_k,b]_\rho)\rtimes \pi')\cdot \dim_\C\ho_{\sp_{m+n}}(\Z(\De)\rtimes\pi,\Z(\De)\rtimes \pi)+\]\[+\dim_\C\ho_{\sp_n}(\sigma,\Z([s_1,b]_\rho)\times \ldots \times \Z([s_k,b]_\rho)\rtimes \pi')\cdot \dim_\C\ho_{\sp_{m+n}}(\Z(\De)\rtimes\pi,\Z(\De)\rtimes \sigma)\le\]\[\le \dim_\C\ho_{\sp_{m+n}}(\Z(\De)\rtimes\pi,\Z(\De)\times \Z([s_1,b]_\rho)\times \ldots \times \Z([s_k,b]_\rho)\rtimes \pi').\]
    Since this is bounded above by  \[2\dim_\C\ho_{\gl_m\times \sp_{n}}(\pi,\Z([s_1,b]_\rho)\times \ldots \times  \Z([s_k,b]_\rho)\rtimes \pi')\] it follows that $\dim_\C\ho_{\sp_{m+n}}(\Z(\De)\rtimes\pi,\Z(\De)\rtimes \pi)=1$.
    and hence $\Z(\De)\rtimes\pi$ is irreducible. 
    We now let $\Pi$ be the maximal $(\Z(\De),\gl_m)$-invariant quotient of $r_{P_m}(\Z(\De)\rtimes\pi)$ and hence by Frobenius reciprocity there exists an embedding
    \[\iota\colon\Z(\De)\rtimes\pi\hra \Z(\De)\rtimes\Pi.\] If we establish that the intertwining operator $I_{\Z(\De),\Pi}$ factors through this embedding, we obtain the claim completely analogously as in the proof of \cite[Theorem 4.1.D.]{LapMin25}. To do so, we consider the following commutative diagram, \emph{cf}. the proof of \cite[Theorem 4.1.D.]{LapMin25}.
    \[\begin{tikzcd} \Z(\De)\times\Z(\De)\rtimes\pi\arrow[rr,"\lambda I_{\Z(\De),\Z(\De)\rtimes\pi}"]\arrow[d,"1_{\Z(\De)}\rtimes\iota"]&&\Z(\De)\times\Z(\De)\rtimes\pi\arrow[d,"1_{\Z(\De)}\rtimes\iota"]\\\Z(\De)\times\Z(\De)\rtimes\Pi\arrow[rr,"I_{\Z(\De),\Z(\De)\rtimes\Pi}"]&&\Z(\De)\times\Z(\De)\rtimes\Pi
    \end{tikzcd}\]
    for suitable $\lambda\in \C$.
    Since $I_{\Z(\De),\Z(\De)}$ is a scalar, we obtain by standard properties of the intertwining operator, \emph{cf.} \cite[§IV]{Wal03}, that the bottom arrow is of the form
    \[\id_{P_{m,m}}^{\sp_{n+2m}}(1_{\Z(\De)}\otimes I_{\Z(\De),\Pi}),\]
    where $1_{\Z(\De)}$ is the map on the second copy of $\gl_m$ in the Levi subgroup, \emph{i.e.}
    \[\gl_m\times \overbrace{\gl_m}^{1_{\Z(\De)}}\times \sp_n.\]
    This implies that \[\Z(\De)\times \Z(\De)\rtimes \pi\hra \id_{P_{m,m}}^{\sp_{n+2m}}({\Z(\De)}\otimes I_{\Z(\De),\Pi}^{-1}(\Z(\De)\rtimes\pi)),\] where again $\Z(\De)$ is the representation on the second copy of $\gl_m$.
    Following the strategy of the proof of \cite[Theorem 4.1.D.]{LapMin25}, we now obtain by \cite[Lemma 2.1]{MinLa18}, see also the proof of \cite[Corollary 2.2]{MinLa18}, that $\Z(\De)\rtimes\Pi=I_{\Z(\De),\Pi}^{-1}(\Z(\De)\rtimes\pi)$, and hence the claim follows.
\end{proof}
\begin{lemma}\label{L:casenotneeded}
    Let $\rho\in \cuu$, $\pi\in \irr(\sp_n)$ and $\De=[a,b]_\rho$ a segment such that one of the following holds.
        \begin{enumerate}
        \item $a+b>0$,
        \item $b< 0$,
        \item $a+b\le 0$ and $a+b\notin\ZZ$.
    \end{enumerate}
    Then \[\dim_\C\ho_{\sp_{m+n}}(\Z(\De)^\lor\rtimes\pi,\Z(\De)\rtimes\pi)=1\] and every non-zero morphism factors through a, necessarily unique, irreducible subrepresentation $\tau\hra\Z(\De)\rtimes\pi$. Moreover, if $\pi'\in\irr(\sp_n)$ is a second representation with $\tau\hra\Z(\De)\rtimes\sigma$, then $\pi\cong \pi'$.
\end{lemma}
\begin{proof}
We first deal with the case where $a=0$ and $b>0$.
    We write $\pi\hra \Z([s_1,b]_\rho)\times \ldots \times \Z([s_k,b]_\rho)\rtimes \pi'$ with $0\le s_1\le s_2\le\ldots\le b$ and there exists no $\ain{s'}{0}{b}$ and suitable $\pi''$ with $\pi\hra\Z([s',b]_\rho)\rtimes\pi''$. We write $r=\deg(\Z([s_1,b]_\rho))+ \ldots + \deg(\Z([s_k,b]_\rho))$.

    Since $b\ge 1$ and $a=0$, the Geometric Lemma and \Cref{T:zel} give as in the proof of \Cref{L:2dim} that   \[\dim_\C\ho_{\sp_n}(\pi,\Z([s_1,b]_\rho)\times \ldots \times \Z([s_k,b]_\rho)\rtimes \pi')\cdot \dim_\C\ho_{\sp_{m+n}}(\Z(\De)^\lor\rtimes\pi,\Z(\De)\rtimes \pi)\le\]\[\le  \dim_\C\ho_{\sp_{m+n}}(\Z(\De)^\lor\rtimes\pi,\Z(\De)\times \Z([s_1,b]_\rho)\times \ldots \times \Z([s_k,b]_\rho)\rtimes \pi')\le\]
    \[\dim_\C\ho_{\gl_m\times \sp_{n}}(\Z(\De)^\lor\otimes\pi,\]\[ \overline{r}_{(m,r)}(\Z(\De)\times \Z([s_1,b]_\rho)\times \ldots \times \Z([s_k,b]_\rho))\rtimes \pi').\]
    By \Cref{L:ext} the above is equal to 
      \[\dim_\C\ho_{\gl_m\times \sp_{n}}(\pi,\Z([s_1,b]_\rho)\times \ldots \times  \Z([s_k,b]_\rho)\rtimes \pi').\]
      We thus obtain that $\dim_\C\ho_{\sp_{m+n}}(\Z(\De)^\lor\rtimes\pi,\Z(\De)\rtimes\pi)=1$. By the MVW-involution it thus follows that $\Z(\De)\rtimes\pi$ has an irreducible socle $\sigma$ and we have a non-zero morphism
      \[\Z(\De)^\lor\rtimes\pi\sra\sigma\hra \Z(\De)\rtimes\pi.\] The proof that $\sigma\hra\Z(\De)\rtimes\pi'$ implies $\pi'\cong\pi$ follows analogously to the proof of \Cref{L:deruniqueseg}.
    Indeed, one first easily see that $\sigma\hra\Z([s_1,b]_\rho)\times \ldots \times \Z([s_k,b]_\rho)\rtimes \pi'$ and as above, one can show that
    \[\dim_\C\ho_{\sp_n}(\Z([s_1,b]_\rho)^\lor\times \ldots \times \Z([s_k,b]_\rho)^\lor\rtimes \pi',\Z([s_1,b]_\rho)\times \ldots \times \Z([s_k,b]_\rho)\rtimes \pi')=1,\] which implies by the MVW-involution that $\pi\cong\sigma$.

      In all the other cases we let $\rho'\coloneq\rho\nu_\rho^a$ and note that $\rho'$ is not self dual and $\Z(\De)^\lor$ is left-$\rho'$-reduced. The claim then follow easily inductively on the length of $\De$ from \Cref{L:derreksp}. Indeed, we saw in the proof of \Cref{{L:derreksp}} that 
    \[\soc(\Z(\De)\rtimes \pi)=\soc((\rho\nu_\rho^a)^{1+d_{\rho\nu_\rho^a,\mathrm{max}}(\pi)}\rtimes(\soc(\Z({}^-\De)\rtimes \mathcal{D}_{\rho\nu_\rho^a}^{\mathrm{max}}(\pi))).\] 
    The claim then follows the basic properties as laid out in \Cref{S:dergl} and \Cref{S:dersp} and the induction hypothesis.
\end{proof}
\begin{theorem}\label{T:segint}
    Let $\rho\in \cuu$, $\pi\in \irr(\sp_n)$ and $\De=[a,b]_\rho$ a segment. Then 
    \[\cos(\im(I_{\Z(\De),\pi}))\cong\soc(\Z(\De)\rtimes\pi).\]
\end{theorem}
\begin{proof}
    We argue by induction on $n$ and note first that if one of the followig conditions is satisfied, the claim is settled by \Cref{L:casenotneeded}. 
            \begin{enumerate}
        \item $a+b>0$,
        \item $b< 0$,
        \item $a+b\le 0$ and $a+b\notin\ZZ$.
    \end{enumerate}
    
    We thus assume otherwise, \emph{i.e.} that $a+b\le 0$, $b>0$ and $a+b\in\ZZ$. In this case we are going firstly to reduce to the case where $a+b=0$.

    Denote $\rho'\coloneq \rho\nu_\rho^a$ and $d\coloneq d_{\rho',\mathrm{max}}(\pi)$. By the assumption on $a$ and $b$, $\rho'$ is not self dual. Applying \Cref{L:interop}(iii) with $\pi_1=\Z(\De),\, \pi_2=\rho'^{\times d}$ and $\tau=\mathcal{D}_{\rho'}^{\mathrm{max}}(\pi)$, we note that by \Cref{L:derivativesint}, that by the lower side of the commutative diagram 
    \[\im(I_{\pi_1,\pi_2,\tau})=\im(I_{\Z(\De)\times\rho'^{\times d},\mathcal{D}_{\rho,\max}(\pi)}).\] Thus the upper side also composes to $I_{\pi_1,\pi_2,\tau}$, and its image is the one of $I_{\Z(\De),\tau}$. 
    Next we apply \Cref{L:interop}(iii) with $\pi_1=\rho'^{r+1},\,\pi_2=\Z({}^-\De)$ and $\tau=\Dd(\pi)$. By the assumption on $\rho'$ and \Cref{L:interop}(i) the upper side of the commutative diagram is then non-zero and has by the induction hypothesis image equal to \[\soc(\rho'^{r+1}\rtimes \mathrm{Im}(I_{\Z({}^-\De), \mathcal{D}_{\rho'}^{\mathrm{max}}(\pi)})).\]
    By \Cref{L:derivativesint}, we have that 
    \[\cosoc(\soc(\rho'^{r+1}\rtimes \mathrm{Im}(I_{\Z({}^-\De), \mathcal{D}_{\rho'}^{\mathrm{max}}(\pi)})))=\soc(\rho'^{r+1}\rtimes \cosoc(\mathrm{Im}(I_{\Z({}^-\De), \mathcal{D}_{\rho'}^{\mathrm{max}}(\pi)}))),\] which by the induction hypothesis is equal to \[\soc(\rho'^{r+1}\rtimes \soc(\Z({}^-\De)\rtimes \mathcal{D}_{\rho'}^{\mathrm{max}}(\pi))).\]
    
    On the other hand, the lower side of the commutative diagram has image equal to \[\im(I_{\Z(\De)\times\rho'^{\times d},\mathcal{D}_{\rho'}^{\mathrm{max}}(\pi)})\] by \Cref{T:zel}.
    Finally, by \Cref{L:derreksp} \[\soc(\rho'^{r+1}\rtimes \soc(\Z({}^-\De)\rtimes \mathcal{D}_{\rho'}^{\mathrm{max}}(\pi)))=\soc(\Z(\De)\rtimes \pi).\]
    This finishes the claim in the case $\rho^\lor\nu_\rho^{-a}\ncong \rho\nu_\rho^b$.

Assume now that $\rho^\lor\nu_\rho^{-a}\cong \rho\nu_\rho^b$ and note that this is equivalent to $\Z(\De)^\lor\cong \Z(\De)$. By \Cref{L:2dimstronger} we have the following cases
\[ \dim_\C \ho_{\sp_{n+m}}(\Z(\De)\times\pi,\Z(\De)\rtimes\pi)=\]\[=\begin{cases}
    2&\text{ if }\Z(\De)\times\pi \text{ is semi-simple of length 2,}\\
        2&\text{ if }\Z(\De)\times\pi \text{ is socle- and cosocle-irreducible, but not irreducible}\\
    1&\text{ if }\Z(\De)\times\pi \text { is irreducible.}
\end{cases}\]
If $\Z(\De)\times\pi$ is semi-simple or irreducible, $I_{Z(\De),\pi}$ acts by a scalar on each summand and for at least one it has to be non-zero. But then \Cref{L:interop} implies that it is non-zero for each summand and hence the image of $I_{Z(\De),\pi}$ equals to all of $\Z(\De)\times \pi$.

Finally, in the second case we denote $\tau=\soc(\Z(\De)\times\pi)=\cosoc(\Z(\De)\times\pi)$. If the intertwining operator is surjective, the cosocle of the image is $\tau$. 
Otherwise, the intertwining operator has to be of the form 
\[\Z(\De)^\lor\times\pi\sra \tau \hra\Z(\De)\rtimes \pi,\]
in which the second case also follows.
\end{proof}
\begin{corollary}\label{C:secrep}
    Let $\pi,\pi'\in\irr(\sp_n)$ and $\De=[a,b]_\rho$ a segment with $\rho\in\cuu$. Then $\soc(\Z(\De)\rtimes\pi)$ is multiplicity-free and of length at most $2$, and if $\tau$ is an irreducible subrepresentation of $\Z(\De)\rtimes\pi$ and $\Z(\De)\rtimes\pi'$, then $\pi\cong \pi'$. 

    The socle of $\Z(\De)\rtimes\pi$ is irreducible if one of the following holds.
    \begin{enumerate}
        \item $a+b>0$,
        \item $b< 0$,
        \item $a+b\le 0$ and $a+b\notin\ZZ$.
    \end{enumerate}\end{corollary}
\begin{proof}
    We argue by induction on $l(\De)$. Note first that if $\De$ is self dual, the claim follows readily from \Cref{L:2dimstronger}. Otherwise, we saw in the proof of \Cref{{L:derreksp}} that 
    \[\soc(\Z(\De)\rtimes \pi)=\soc((\rho\nu_\rho^a)^{1+d_{\rho\nu_\rho^a,\mathrm{max}}(\pi)}\rtimes(\soc(\Z({}^-\De)\rtimes \mathcal{D}_{\rho\nu_\rho^a}^{\mathrm{max}}(\pi))).\] 
    The claim then follows the basic properties as laid out in \Cref{S:dergl} and \Cref{S:dersp} and the induction hypothesis.
    If one of the three conditions is satisfied, the claim follows from \Cref{L:casenotneeded}. 
\end{proof}
Finally let us make the following definitions.
Let $n\le m\in\NN$. For $\pi\in\irr_n$, let $\ain{r}{0}{n}$ be minimal such that there exists $\tau\in\irr_r$ with $\pi\hra\abs{n-r}{\frac{r}{2}}\times\tau$. Then \[ \mathfrak{T}_n^m(\pi)\coloneq\mathrm{soc}(\tau^\lor\nu^{\frac{m-n}{2}}\times \abs{m-r}{-\frac{r}{2}}).\] By \cite[Lemma 2.8]{MinLa18} this is well defined and an irreducible representation.

Let $\pi\in \irr(\sp_n)$ and let $\mu$ be a unitary character of $\gl_1$.
We let $\ain{r}{0}{n}$ be maximal such that there exists $\tau\in\irr(\sp_{n-r})$ with $\pi\hra\mu(\det_r)^{-1}\abs{r}{-\frac{m-n+r}{2}}\rtimes\tau$. Then \[ \mathfrak{T}_n^m(\pi,\mu)\coloneq\mathrm{soc}(\mu(\det_{m-n+r})^{-1}\abs{m-n+r}{-\frac{r}{2}}\rtimes\tau^\lor).\] By \Cref{C:secrep} this is well defined and is a multiplicity-free representation of length at most $2$.
\subsection{Intertwining operators for general linear groups}
In this subsection we will demonstrate how one is able to compute poles of certain intertwining operators for general linear groups explicitly. We start with the easiest case.
\begin{lemma}\label{L:simplepoles}
\[\Lambda(\rho,\rho')=\begin{cases}
    1&\text{if }\rho'\cong \rho,\\0&\text{otherwise}.
\end{cases}.\]
\end{lemma}
\begin{proof}
    For the computation of $\Lambda$ see for example \cite[Proposition 7.5]{Dat}.
\end{proof}
To continue, we prove next the following.
\begin{lemma}\label{L:polevanish}
    Let $\pi_1,\pi_2\in \Irr$ such that $\pi_1\otimes \pi_2$ appears in \[r_{\deg(\pi_1),\deg(\pi_2)}(\pi_2\times \pi_1)\] with multiplicity $1$, then $\Lambda(\pi_1,\pi_2)=0$.
\end{lemma}
\begin{proof}
    Note that by Frobenius reciprocity, the condition of the lemma implies that 
there exists a, up to a scalar unique, morphism $T\colon \pi_2\times\pi_1\ra \pi_1\times \pi_2$ with the defining property that for $f\in \pi_2\times \pi_1$ with support contained in $P_{\deg(\pi_2),\deg(\pi_1)}w_{\deg(\pi_2),\deg(\pi_1)}P_{\deg(\pi_1),\deg(\pi_2)}$,
\[T((f)(1)=\int_{U_{\deg(\pi_1),\deg(\pi_2)}}f(w_{\deg(\pi_2),\deg(\pi_1)}u)\,\mathrm{d}u.\]
By the definitions of \cite[§IV]{Wal03}, this implies that $\Lambda(\pi_1,\pi_2)=0$.
\end{proof}
\begin{lemma}\label{L:central}
Let $\De=[0,b]_\rho$ be a segment and $\pi_1,\pi_2\in \Irr$ such that for all $\ain{a}{0}{b}$ we have that $\pi_1\times \Z([a,b]_\rho)$ is irreducible and $\pi_1, \Dde(\pi_1)\in \irs$. Then \[\Lambda(\pi_1,\pi_2)=\Lambda(\Dde(\pi_1),\Dde(\pi_2))+\Lambda(\pi_1,\De(\pi_2)).\]
\end{lemma}
\begin{proof}
    By \Cref{L:interop}, there exists $\lambda\in \C$ such that the following diagram commutes
\[\begin{tikzcd}
    \De(\pi_2)\times \Dde(\pi_2)\times\pi_1\arrow[rr,"I_{\pi_1,\De(\pi_2)\times\Dde(\pi_2)}"]&&\pi_1\times \De(\pi_2)\times \Dde(\pi_2)\\
    \pi_2\times\pi_1\arrow[rr,"\lambda I_{\pi_1,\pi_2}"]\arrow[u,hookrightarrow]&&\pi_1\times \pi_2\arrow[u,hookrightarrow]
\end{tikzcd}\]
    and $\lambda\neq 0$ if and only if \[\Lambda(\pi_1,\pi_2)=\Lambda(\pi_1,\Dde(\pi_2))+\Lambda(\pi_1,\De(\pi_2)).\]
    By \Cref{L:prodsquare}, $\pi_1\times \De(\pi_2)\in \irs$, thus the socle $\sigma\coloneq\soc(\pi_1\times \De(\pi_2)\times \Dde(\pi_2))$ is irreducible. Hence $\sigma=\soc(\pi_1\times\pi_2)$ and if $\lambda=0$, it would appear as a subrepresentation with multiplicity $2$ in $\pi_1\times \De(\pi_2)\times \Dde(\pi_2)$, a contradiction to \Cref{T:si}.
    It remains to show that $\Lambda(\pi_1,\Dde(\pi_2))=\Lambda(\Dde(\pi_1),\Dde(\pi_2))$. To see this, we will use the commutative diagram
    \[\begin{tikzcd}
    \Dde(\pi_2)\times\De(\pi_1)\times\Dde(\pi_1)\arrow[rrr,"I_{\De(\pi_1)\times\Dde(\pi_1),\Dde(\pi_2)}"]&&&\De(\pi_1)\times\Dde(\pi_1)\times\Dde(\pi_2)\\
    \Dde(\pi_2)\times\pi_1\arrow[rrr,"\lambda' I_{\De(\pi_1,\Dde(\pi_2)}"]\arrow[u,hookrightarrow]&&&\pi_1\times \Dde(\pi_2)\arrow[u,hookrightarrow]
\end{tikzcd}\] for a suitable $\lambda'\in \C$. We start by showing that $\lambda'\neq 0$. For this, it is, as in the first step, enough to show that the socle $\soc(\De(\pi_1)\times\Dde(\pi_1)\times\pi_2)$ is irreducible, which by \Cref{T:si} can be shown by proving that there exists a, up to a scalar, unique morphism 
\[\Dde(\pi_2)\times\Dde(\pi_1)\times \De(\pi_1)\ra \De(\pi_1)\times \Dde(\pi_1)\times\Dde(\pi_2).\]
By Frobenius reciprocity, the Geometric Lemma, and the fact that both $\Dde(\pi_1)$ and $\Dde(\pi_2)$ are left-$\De$-reduced, we obtain that any such morphism induces a morphism
\[\Dde(\pi_2)\times\Dde(\pi_1)\ra \Dde(\pi_1)\times\Dde(\pi_2).\] The claim then follows from \Cref{T:si}. We therefore showed that 
 \[\Lambda(\pi_1,\Dde(\pi_2))=\Lambda(\Dde(\pi_1),\Dde(\pi_2))+\LL(\De(\pi_1),\Dde(\pi_2)).\]
 By \Cref{L:polevanish}, $\LL(\De(\pi_1),\Dde(\pi_2))$ vanishes and the claim follows.
\end{proof}
Applying \Cref{L:central} and \Cref{L:polevanish} to the case $\pi_1\cong \Z(\De_1)$ and the segment $\De=[0,0]_\rho$ yields the following.
\begin{corollary}\label{C:seg}
    Let $\pi_1=\Z(\De_1)$, $\De_1=[0,b]_\rho$, $\rho\in \cus$. Then for any $\pi\in \Irr$, we have
    \[\Lambda(\Z(\De_1),\pi)=\Lambda(\Z({}^-\De_1),\mathcal{D}_{\rho,l}^{\mathrm{max}}(\pi))\] if $l(\De)>1$ and
    \[\Lambda(\rho,\pi)=\dlm(\pi).\]
\end{corollary}
As a corollary, we obtain the following.
\begin{corollary}\label{C:onesegment}
    Let $\De=[0,b]_\rho$ be a segment, $k\in \NN$ and $\ain{a_i}{0}{b}$, $\ain{i}{1}{k}$. Then
    \[\Lambda(\Z(\De),\Z(\De_1)\times\ldots\times \Z(\De_k))=k.\]
\end{corollary}
From \Cref{L:central} we are thus able to deduce.
\begin{theorem}\label{T:centralint}
    Let $\De$ be a segment and $\pi\in \irr$. Then $\Lambda(\Z(\De),\pi)=l_k(\De)$.
\end{theorem}
\section{Spherical varieties}\label{S:spherical}
Let $\bG$ be a split-reductive group with Borel subgroup $\bB$, $\bT$ the maximal torus in $\bB$ and $\bU$ its unipotent part, and $\bX$ a $\bG$-spherical variety, both over $\Ff$, \emph{i.e.} $\bX$ is normal and admits an open (and dense) $\bB$-orbit, \emph{cf.} \cite{BriLunVus86}, \cite{Bri86}. We assume further that for any parabolic subgroup $\bB\subseteq \bP\subseteq \bG$, the $\bP$-orbits are defined over $\Ff$ and are orbits of $P$. We only assume this for the sake of simplicity and since all explicit cases we will encounter satisfy this property. For a more nuanced discussion of problems relating to the rationality of these orbits, we refer the reader for example to \cite[§3]{Sak08}.

If $\bX$ is homogeneous, \emph{i.e.} $\bX=\bH\backslash\bG$ for a subgroup $\bH\subseteq \bG$, $\bH$ is called a spherical subgroup of $\bG$ and any $\bG$ spherical variety $\bX'$ with $\bH\bs\bG\subseteq \bX'$ is called a spherical embedding of $\bH\bs\bG$. For example, if $\bN$ denotes the unipotent part of a parabolic subgroup of $\bG$, $\bN$ is a spherical subgroup by virtue of the Bruhat decomposition. Similarly, $\Delta \bG\subseteq \bG\times \bG$ is a spherical subgroup of $\bG\times\bG$, which is often referred to as the \emph{group case}.

We recall that in general $\bX$ admits only a finite number of $\bB$-, and hence $\bG$-, orbits and any $\bG$-stable closed subvariety of $\bX$ is again a spherical variety. We denote by $\oo(\bX)$ the set of $\bG$-orbits of $\bX$ and the closure-order on $\oo(\bX)$ by $\le$, \emph{i.e.} $\bY\le \bY'$ if and only if $\bY\subseteq \overline{\bY'}$ if and only if $\bY(\Ff)\subseteq \overline{\bY'(\Ff)}$.
One reason why spherical varieties are such a fruitful place to look for interesting phenomena is that they admit combinatorial description in terms of colored fans.
\subsection{Colored fans}
We will now recap rather quickly and superficially the theory of colored fans and the classification of spherical embeddings. For a more thorough exposition to this topic see for example \cite{Per14}.
Let $\bH\subseteq \bG$ be a spherical variety and $\bX$ a $\bH\backslash\bG$-spherical embedding.
Let $\Lambda(\bX)$ be the weight lattice of $\bX$, \emph{i.e.} the set of weights $\lambda\in \mathbb{X}^*(\bT)$ such that there exists a non-zero $\bU$-invariant, $(\lambda,\bT)$-invariant $f\in \Ff(\bX)$. Let $\mathcal{O}(\bX)\coloneq \ho_\ZZ(\Lambda(\bX),\mathbb{Q})$ and define the rank of $\bX$ as $\dim_\mathbb{Q}\mathcal{O}(\bX)$. Furthermore, we need to consider valuations on $\bH\bs\bG$, \emph{i.e.} functions \[\nu\colon\Ff(\bH\backslash\bG)\ra \mathbb{Q}\cup\{\infty\}\] satisfying some natural axioms.
We denote by $\mathcal{V}(\bH\backslash\bG)$ the set of $\bG$-invariant valuations.
There exists then a natural injective map $\mathcal{V}(\bH\backslash\bG)\hra \mathcal{O}(\bH\backslash\bG)$, given by $\nu\mapsto (\lambda\mapsto \nu(f_\lambda))$, where $f_\lambda\in \Ff[\bX]$ is a $\lambda$-eigenfunction of $\bB$.
We define
\[D(\bX)\coloneq \{\bD:\bB\text{-stable divisor of }\bX\},\, \Delta(\bX)\coloneq \{\bD:\bB\text{-stable, not }\bG\text{-stable divisor of }\bX\}.\]
The elements of $\Delta(\bX)$ are called the \emph{colors} of $\bX$ and note that $\Delta(\bX)\subseteq \Delta(\bH\backslash\bG)$. Moreover there exists a natural inclusion \[D(\bX)\setminus\Delta(\bX)\hra \mathcal{V}(\bH\bs\bG)\]
given by sending a divisor $\bD$ to the valuation sending a function to the order of its pole/zero along $\bD$.
For $\bY\in\oo(\bX)$
\[D_\bY(\bX)\coloneq \{\bY\subseteq\bD:\bB\text{-stable divisor of }\bX\},\,\Delta_\bY(\bX)\coloneq D_\bY(\bX)\cap \Delta(\bX).\]
We also have an inclusion
\[\Delta_\bY(\bX)\hra \mathcal{O}(\bH\backslash\bG)\]
 given by $\bD\mapsto (\lambda\mapsto \nu_\bD(f_\lambda))$, where $f_\lambda\in \Ff[\bX]$ is a $\lambda$-eigenfunction of $\bB$ and $\nu_\bD$ is the valuation associated to the divisor $\bD$.
Finally, we let $\con(\bX)$ be the cone generated by $D(\bX)\setminus\Delta(\bX)$ and $\Delta(\bX)$ in $\mathcal{O}(\bH\backslash\bG)$, and for $\bY\in\oo(\bX)$, $\con_\bY(\bX)$ be the cone generated by $D_\bY(\bX)\setminus\Delta_\bY(\bX)$ and $\Delta_\bY(\bX)$ in $\mathcal{O}(\bH\backslash\bG)$

A \emph{colored cone} $(\curC,\curF)$ of $\bH\backslash\bG$ consists of a finite subset $\curC\subseteq \Delta(\bH\bs\bG)$ and a cone $\curF$ in $\mathcal{O}(\bH\backslash\bG)$ generated by a finite number of points in $\mathcal{V}(\bH\bs\bG)$ and $\curC$ such that $\curF^\circ\cap \mathcal{V}(\bH\bs\bG)$ is non-empty. A \emph{colored face} of a colored cone $(\curC,\curF)$ of $\bH\bs\bG$ is a colored cone $(\curC',\curF')$ of $\bH\bs\bG$ such that $\curF$ is a face of $\curF'$, $\curC'=\curC\cap \curF'$ and ${\curF'}^\circ\cap \mathcal{V}(\bH\bs\bG)$ is non-empty.
A \emph{colored fan} of $\bH\bs\bG$ is a finite set $\mathfrak{F}=\{(\curC,\curF)\}$ of colored cones of $\bH\bs\bG$ such that any colored face of a colored cone in $\mathfrak{F}$ is again in $\mathfrak{F}$ and every $v\in \mathcal{V}(\bH\bs\bG)$ is contained in at most one colored cone in $\mathfrak{F}$. A colored cone $(\curC,\curF)$ is called strictly convex if $0\notin \curF$ and a colored fan $\mathfrak{F}$ is called strictly convex if $(0,\emptyset)\in \mathfrak{F}$.
\begin{theorem}[{Luna-Vust theory, \cite{LunVus83}}]\label{T:classpher}
    The map
    \[\{\bH\bs\bG\text{-spherical embeddings}\}_{/\cong}\ra \{\text{strictly colored fans}\},\]
    \[\bX\mapsto \{(\con_\bY(\bX),\Delta_\bY(\bX)):\bY\in\oo(\bX)\}\]
    is well defined and a bijection.
\end{theorem}
In particular, if $\bX$ is ordered with $r+1$ orbits and has rank $r$, the cone $\con(\bX)$ has a unique face of dimension $i$, $\ain{i}{1}{r}$, whose interior intersects $\mathcal{V}(\bH\bs\bG)$ non-trivially.
\subsection{Local Structure Theorem}
For $\bY\in\oo(\bX)$ we denote by
\[\bX^\circ_{\bY}\coloneq \bX-\bigcup_{D_{\bY}(\bX)}\bD=\{x\in X: \bY\subseteq \overline{\bB x}\}.\]
Let $\bP_\bY$ be the stabilizer of $\bX_\bY$, which is a parabolic subgroup with Levi-subgroup $\bL_\bY$ and unipotent part $\bN_\bY$. For a proof of the following see \cite{BriLunVus86} for algebraically closed fields and \cite{knopKrötz} for the general case.
\begin{theorem}[{Local Structure Theorem}]\label{T:lst}
    Then there exists an $\bL$-spherical, affine variety $\bS_\bY\subseteq \bX^\circ_{\bY}$ over $\Ff$ such that 
    \[\bN_\bY\times \bS_\bY\ra \bX^\circ_{\bY}\] is an isomorphism and $\bX^\circ_{\bY}\cap \bY$ consists of the unique open and dense $\bB$-orbit in $\bY$.
\end{theorem}
For a parabolic subgroup $\bB\subseteq\bP\subseteq \bG$ we let $\oo_\bP(\bX)=\{\bY\in\oo(\bX): \bP_\bY=\bP\}$. 

We call a spherical variety $\bX$ \emph{ordered} if the order-relation on $\oo(\bX)$ is a total order, \emph{i.e.} for two orbits $\bY,\,\bY'$ either $\bY'\le \bY$ or $\bY'\le \bY$.

In our applications of the above theorems, we will be in the following setting.
Let $\bG_1$ be either a symplectic or general linear group, $\bP=\bL\bN$m either a standard maximal parabolic subgroup of $\bG$ or $\bG$ itself with $\bL=\bG_2\times \bL'$. Let $\bX$ be an ordered $\bG_2\times \bG_1$ spherical embedding of $\bG_2\times^{\Delta \bG_2}\bG_1$ such that each $\bY\in\oo(\bX)$ is of the following form.
Either $\bY$ is just a point or there exists a standard maximal parabolic subgroups $\bP_1=\bL_1\bN_1\subseteq \bG_1$, $\bP_2=\bL_2\bN_2\subseteq \bG_2$ such that $\bL_1\cong\bG_3\times \bM_{1}$, $\bL_2\cong\bG_3\times \bM_2$ and 
\[\bY\cong (\bP_1\times^{\Delta\bG_3}\bP_2)\backslash \bG_2\times\bG_1.\]
Moreover, if $\bY'$ is a second orbit with $\bY'\le \bY$, we ask that 
\[\bM_1'\subseteq \bM_1,\, \bM_2'\subseteq \bM_2.\]
\begin{lemma}\label{L:spespher}
Let $\bY\in \oo(\bX)$.
Then $\oo(\bS_\bY)=\{\bS_\bY\cap \bY': \bY'\le \bY\}$, $\bS_\bY$ is ordered and $\bS_\bY\cap \bY'\times\bN_\bY$ is the unique open $\bP_\bY$ orbit in $\bY'$. Then the choice of a Borel subgroup such that $\bB\subseteq \overline{\bP_1\times \bP_2}$ implies that $\bP_\bY=\overline{\bP_1\times \bP_2}$.
\end{lemma}
\begin{proof}
    By the Bruhat decomposition there exists for each $\bY'\in\oo(\bX)$ at most one $\bP_\bY$-orbit in $\bX_\bY^\circ\cap\bY'$ with trivial $\bN_\bY$ stabilizer. Moreover, if it exists, it has to be the open $\bP_\bY$-orbit.
    If $\bY'<\bY$, $\bS_\bY\cap \bY'=\emptyset$. On the other hand, if $\bY'>\bY$, it is clear that $\bY\subseteq \overline{\bY}\subseteq \overline{\bP\bY'}$, hence $\bY'\cap \bS_\bY\neq \emptyset$.
    The first part of the lemma then easily follows.

    For the second part, we argue as follows. It is clear that the union of the open $\bP'\coloneq\overline{\bP_1\times \bP_2}$-orbits in the orbits $\bY'\le \bY$ is contained in $\bX_{\bY}^\circ$ and therefore $\bP'\subseteq \bP_\bY,\, \bN_\bY\subseteq \bN'$. On the other hand, $\bP_\bY\subseteq\bP'$, since $\bP_\bY$ has to stabilize the open $\bB$-orbit in $\bY$. But this implies that $\bN'\subseteq\bN_\bY$ and hence the claim follows.
\end{proof}
\subsection{Representation theory of spherical varieties}\label{S:sphericalrep}
In this section we will now apply the above theorem to the representation theory of $\bX$-distinguished representations.
Let $\bX$ and $\bG$ be as above and consider the smooth $G$-representation \[S(\bX)=S(\bX(\Ff))\coloneq\{\phi\colon \bX(\Ff)\ra \C:\, \phi\text{ is locally constant and compactly supported}\},\] on which $G$ acts by translation. We also consider the twisted versions $S(\bX,\LL)$, where $\LL$ is an $\ell$-sheaf and $S(\bX,\LL)$ denotes the global, compactly supported, sections of $\LL$, see \cite{BerZel76}.
We denote by $\irr_{\bX,\LL}(\bG)$ the set of $(\bX,\LL)$ distinguished representations or the spectrum of $(\bX,\LL)$, \emph{i.e.} 
\[\irr_{\bX,\LL}(\bG)\coloneq\{\pi\in\irr(G):\ho_{G}(S(\bX,\LL),\pi)\neq 0\}.\]
We call $(\bX,\LL)$ \emph{multiplicity-free} if for all $\pi\in \irr_{\bX,\LL}(\pi)$ we have
\[\dim_\C\ho_{G}(S(\bX,\LL),\pi)= 1.\]
In particular we have in the group case that $\Delta\bG\backslash\bG\times \bG$ is multiplicity free and \[\irr_{\Delta\bG\backslash\bG\times \bG}(\bG)=\{\pi\otimes\pi^\lor:\pi\in \irr(G)\}.\]

If $\bB\subseteq \bP$ is a maximal parabolic subgroup of $\bG$ with Levi-component $L$ and unipotent part $N$, $\sigma$ in $\irr(L(\Ff))$ and $\pi\in \irr(G)$, we call $\pi$ $(\bX,\LL,\bP,\sigma)$-\emph{generic} if $\pi\hra\id_\bP^\bG(\sigma)$ and the following holds:
For any non-zero morphism
$S(\bX,\LL)\ra \pi$, the map obtained by Frobenius reciprocity
\[r_{\bN(\Ff)}(S(\bX,\LL))\ra\sigma \]does not vanish on
\[r_{\bN(\Ff)}(S(\bP x_0,\restr{\LL}{\bP x_0}),\]
where $\bP(\Ff)x_0$ is the open $\bP(\Ff)$-orbit of $\bX(\Ff)$.
Note that by \Cref{T:lst} we have that for $\bY\in\oo_{\bP}(\bX)$,
\[\delta_{P_Y}^{\frac{1}{2}}S(\bS_{\bY},\restr{\LL}{\bS_{\bY}})\hra r_{\bN(\Ff)}(S(\bX,\LL)).\]
In particular if $\pi\in\irr_{\bX,\LL}(G)$ is $(\bX,\LL,\bP,\sigma)$-generic and for all $\bY'\in\oo(\bX,\pi)$ we have $\bY'\cap\bS_\bY\neq\emptyset$, we note that the composition
\[\delta_{P_Y}^{\frac{1}{2}}S(\bS_{\bY},\restr{\LL}{\bS_{\bY}})\hra r_{\bN(\Ff)}(S(\bX,\LL))\sra \sigma\]
obtained by Frobenius reciprocity is non-zero.

Finally, we say $(\bX,\LL)$ has the \emph{lifting property for }$\pi\in \irr(G)$ if for all $\bY\in\oo(\bX)$ the following condition holds.
If $\pi\in\irr_{\bY,\restr{\LL}{\bY}}(G)$ and
$S(\bX,\LL)\ra \pi$ is a morphism, the composition
\[S(\bX',\restr{\LL}{\bX'})\hra S(\bX,\LL)\ra\pi, \, \bX'\coloneq\bigcup_{\bY'>\bY}\bY'\]
vanishes.

For $\pi\in \irr_{\bX,\LL}(G)$, let $\oo(\bX,\pi)$ be the set of orbits $\bY$ such that $\pi\in\irr_{\bY,\restr{\LL}{\bY}}(G)$ and $\oo(\bX,\pi)^{<}$ be the subset of $\oo(\bX,\pi)$ consisting of those orbits $\bY$ for which there exists $\bY>\bY'$ with $\bY'\in \oo(\bX,\pi)$. 
\begin{theorem}\label{T:lstcor1}
    Let $\pi\in \irr(G)$ and assume the following.
    \begin{enumerate}
        \item There exists a maximal parabolic subgroup $\bB\subseteq \bP$ such that $\pi$ is $(\bY,\restr{\LL}{\bY}\LL,\bP,\sigma)$-generic for all $\bY\in\oo(\bX,\pi)^{<}$.
        \item $\oo(\bX,\pi)\setminus\oo(\bX,\pi)^{<}\subseteq \oo_{\bP}(\bX)$.
        \item For all $\bY\in \oo(\bX,\pi)$ with $\bP_\bY=\bP$, $S(\bS_{\bY},\restr{\LL}{\bS_{\bY}})$ has the lifting property for $\delta_{P_Y}^{-\frac{1}{2}}\sigma$.
    \end{enumerate}
    Then $(\bX,\LL)$ has the lifting property for $\pi$.
\end{theorem}
\begin{proof}
Let $\pi\in\irr_\bX(G)$ and fix a nonzero morphism $f:S(\bX,\LL)\ra \pi$.
 and assume there exists $\bY'> \bY$ such that $\bY\in \oo(\bX,\pi)\setminus\oo(\bX,\pi)^{<}$, $f$ vanishes on $S(\bigcup_{\bY''>\bY'}\bY'',\restr{\LL}{\bigcup_{\bY''> \bY'}\bY''})=0$ but does not vanish on $S(\bigcup_{\bY''>\bY'}\bY'',\restr{\LL}{\bigcup_{\bY''\ge \bY'}\bY''})$.
By assumption, the non-zero morphism obtained from Frobenius reciprocity
    \[r_{\bN(\Ff)}(S(\bX,\LL))\ra\sigma\] does not vanish on $\delta_{P_Y}^{\frac{1}{2}}S(\bS_{\bY},\restr{\LL}{\bS_{\bY}})$.
    We thus obtain a non-zero map $f'\colon S(\bS_{\bY},\restr{\LL}{\bS_{\bY}})\ra\delta_{P_Y}^{-\frac{1}{2}}\sigma$.
    Since $\bY'>\bY$, we have that $\bX^\circ_{\bY}\cap \bY'$ contains the unique open $\bP$-orbit of $\bY'$, which is hence contained in $(\bP\cap\bS_\bY)\times \bN$. By assumption on $f$, there exists a $\bL$-orbit $\bY_1$ with $\bY_1>\bS_{\bY}\cap \bY$ such that $f'$ vanishes on all sections supported 
    on $\bigcup_{\bY_2>\bY_1}\bY_2$ and does not vanish on the sections supported on $\bigcup_{\bY_2\ge\bY_1}\bY_2$.
    But since $(\bS_{\bY},\restr{\LL}{\bS_{\bY}}))$ has the lifting property for $\sigma$, it follows that $\bY_2<\bS_{\bY}\cap \bY$, a contradiction.
\end{proof}
\begin{corollary}\label{C:liftingproperty}
    Let $\pi\in \irr(G)$ and assume the following.
    \begin{enumerate}
        \item There exists a maximal parabolic subgroup $\bB\subseteq \bP$ such that $\pi$ is $(\bY,\restr{\LL}{\bY}\LL,\bP,\sigma)$-generic for all $\bY\in\oo(\bX,\pi)^{<}$.
        \item $\oo(\bX,\pi)\setminus\oo(\bX,\pi)^{<}\subseteq \oo_{\bP}(\bX)$.
        \item For all $\bY\in \oo(\bX,\pi)$, $(\bY,\restr{\LL}{\bY})$ is multiplicity free.
        \item For all $\bY\in \oo(\bX,\pi)$ with $\bP_\bY=\bP$, $(\bS_{\bY},\restr{\LL}{\bY}))$ has the lifting property for $\delta_{P_Y}^{-\frac{1}{2}}\sigma$.
        \item $\bX$ is ordered.
    \end{enumerate}
    Then $(\bX,\LL)$ is multiplicity-free and has the lifting property for all irreducible representations.
\end{corollary}
\begin{proof}
Let $\pi\in \irr_{\bX,\LL}(G$ and $\bY$ be the minimal orbit in $\oo(\bX,\pi)$.
    By \Cref{T:lstcor1}, we have
    \[\ho_{G}(S(\bX,\LL),\pi)\subseteq \ho_{G}(S(\bY,\restr{\LL}{\bY}),\pi)\cong \C.\]
\end{proof}
\section{The examples \texorpdfstring{$\M_{n,m}$}{Mat{n,m}} and \texorpdfstring{$\Sg_{n,m}$}{SG{n,m}}}
In this section we will apply the results of \Cref{S:spherical} to two examples, namely the space of matrices and the symplectic Grassmannian. Furthermore, we also explain how to extend these results to the respective metaplectic covers. Together with the arguments of \cite{DroKudRa} and \cite{DroHow}, which are based on the same idea, we thus show that the Howe-duality in Type I and II is essentially a consequence of the theory of derivatives and the Local Structure Theorems of the underlying spherical varieties. Moreover, we will be able to explicitly describe the local Miyawaki lifts of \cite{Ato19}.
\subsection{$\M_{n,m}$}
We start with the case $\bX=\M_{n,m}$, the space of $n\times m$-matrices on which $\bG=\Gl_n\times \Gl_m$ acts by left-right translation. We fix the Borel-subgroup $\bB$ of $\bG$ as the upper triangular matrices in both components. For simplicity we will assume from now on that $n\le m$. It is clear that all results presented here have their natural analogues in the case $n\ge m$.

The $\Gl_n\times \Gl_m$-orbits on $\M_{n,m}$ are given by $\M_{n,m}^r\coloneq \{X\in \M_{n,m}:\rk(X)=r\}$, $\ain{r}{0}{n}$, and $\M_{n,m}^r\le \M_{n,m}^p$ if and only if $r\le p$. We set
\[\M_{n,r}^{\ge r}\coloneq\bigcup_{r'\ge r}\M_{n,m}^{r'},\] which is an open subset of $\M_{n,m}^r$.
The orbit $\M_{n,m}^r$ is spanned by the element
\[\eta_r\coloneq \begin{pmatrix}
    0&0&1_r\\0&0&0
\end{pmatrix}\]
and its stabilizer is given by the subgroup \[\bP_{\eta_r}\coloneq \bP_{r,n-r}\times^{\Delta \Gl_r}\bP_{m-r,r}\subseteq \Gl_n\times \Gl_m.\]
We start with the following structural observation.
\begin{lemma}\label{L:groupglsp}
    Let $\bX$ be an ordered affine $\Gl_n\times \Gl_m$ spherical variety with \[\oo(\bX)=\{\M_{n,m}^r,\ldots\M_{n,m}^0\}\] for some $\ain{r}{0}{n}.$ Then $\bX\cong \M_{n,m}^{\le r}$.
\end{lemma}
\begin{proof}
    We argue by induction on $r$, the case $r=0$ being trivially true, and will use \Cref{T:classpher}. 
    We note that $\M_{n,m}^r$ and $\M_{n,m}^{\le r}$ is of rank $r$. Indeed, by the induction hypothesis $\bX\setminus \M_{n,m}^r\cong \M_{n,m}^{\ge r-1}$ has rank $r-1$ and hence $\bX$ has to have rank $r$. Our choice of Borel-groups gives us a natural identification of $\mathcal{O}(\M_{n,m}^r)\cong \mathbb{Q}^r$, see for example \cite[Example 5.2.5]{Per14}. 
    In this coordinates $D(\M_{n,m}^r)$ correspond to $\epsilon_k-\epsilon_{k+1},\,\ain{k}{1}{r}$, where $\epsilon_i$ is the $i$-th unit vector for $i\le r$ and $\epsilon_{r+1}=0$. Unless $n=m=r$, $\De(\M_{n,m}^r)=D(\M_{n,m}^r)$, and if $n=m=r$, $\De(\M_{n,m}^r)=D(\M_{n,m}^r)\setminus\{\epsilon_1-\epsilon_2\}$.
    The valuation cone $\mathcal{V}(\M_{n,m}^r)$ corresponds then to $\{(\lambda_1,\ldots,\lambda_r)\in \ZZ^r:\lambda_1\ge\ldots\ge \lambda_r\}$. By the induction hypothesis we know that $\bX$ contains $\M_{n,m}^{\le r-1}$ as a maximal proper $\Gl_n\times \Gl_m$-stable subvariety, and since $\bX$ is ordered, we know that there exists only one face of $\con(\bX)$ intersecting the interior of $\mathcal{V}(\M_{n,m}^r)$, which must in turn correspond to the cone of $\M_{n,m}^{\le r-1}$, and hence in our coordinates corresponds to the cone spanned by $\epsilon_k-\epsilon_{k+1},\,\ain{k}{1}{r-1}$. Since the rank of $\bX$ is $r$, we need to find at least one more generator of $\con(\bX)$ which does not come from $\M_{n,m}^{\le r-1}$. It cannot lie in $\mathcal{V}(\M_{n,m}^r)$, since otherwise $\bX$ would not be ordered, thus it is forced to be $\epsilon_r$. It follows that $\bX$ is isomorphic to $\M_{n,m}^{\le r}$.
\end{proof}
We then obtain as a consequence of \Cref{L:spespher} the following.
\begin{corollary}\label{L:localstrugl}
    Let $\bP=\bP_{r,n-r}\times \overline{\bP_{r,m-r}}$.
    Then $\M_{m,n}^r\in \oo_{\bP}(\M_{n,m})$ and
    \[\bS_{\M_{n,m}^r}=\M_{n-r,m-r}\times \Gl_r=\{\begin{pmatrix}
        g&0\\0&X
    \end{pmatrix}:g\in\Gl_r,\,X\in\M_{n-r,m-r}\}.\]
\end{corollary}
\subsection{$\Sg_{m,n}$}\label{S:sympgra}
Next we treat the case of the symplectic Grassmannian
Let $n\le m\in \NN$ and consider the Siegel-parabolic $\bP_{n+m}\subseteq \Sp_{n+m}$. We denote the symplectic Grassmannian $\Sg_{n,m}\coloneq \bP_{n+m}\backslash\Sp_{n+m}$, the space of $n$-dimensional isotropic subspace of $\AA^{2n+2m}$, where we equip $\AA^{2n+2m}$ on $\AA_+^{2n}$ with the standard symplectic form and on $\AA_-^{2m}$ with the negative standard symplectic form. 
This gives rise to an embedding
\[\Sp_n\times\Sp_m\hra \Sp_{n+m},\]
\[\left(\begin{pmatrix}
    A&B\\C&D
\end{pmatrix},\begin{pmatrix}
    A'&B'\\C'&D'
\end{pmatrix}\right)\mapsto \begin{pmatrix}
    A&0&B&0\\0&A'&0&-B'\\C&0&D&0\\0&-C'&0&D'
\end{pmatrix},\]
which in turn induces a right-action of $\Sp_n\times\Sp_m$ on $\Sg_{n,m}$.
The open and dense $\Sp_n\times \Sp_m$-orbit is given by those isotropic subspaces $U\in \Sg_{n,m}$ such that $U\cap \AA^{2n}$ is $0$-dimensional.
From now on we fix the choice of the Borel subgroup $\bB=\bP_{1,\ldots,1}\times\bP_{1,\ldots,1}$ of $\Sp_n\times \Sp_m$. The stabilizer of any point in the open $\Sp_n\times \Sp_m$-orbit is given by $ \Sp_n\times^{\Delta\Sp_n}\bP_{m}\subseteq \Sp_n\times \Sp_m $ , which is clearly a spherical subgroup, hence $\Sg_{n,m}$ is a spherical variety. 

The $\Sp_n\times \Sp_m$-orbits on $\Sg_{n,m}$ are given precisely by 
\[\Sg_{n,m}^r\coloneq \{U\in \Sg_{n,m}: \dim_\Ff(U\cap \AA_+^{2n})=r\},\,\ain{r}{0}{n}\]
and $\Sg_{n,m}^r\ge \Sg_{n,m}^{r'}$ if and only if $r\le r'$. We set \[\Sg_{n,m}^{\le r}\coloneq \bigcup_{r'\le r}\Sg_{n,m}^{r'},\] which are open subsets of $\Sg_{n,m}$. Following \cite{KudlaRallis}, we let $\delta_0$ be the element in $\Sg_{n,n}^0$ represented by \[\delta_0\coloneq \begin{pmatrix}
    0&0&1_n&0\\0&1_n&0&0\\-1_n&1_n&0&0\\0&0&1_n&1_n
\end{pmatrix}.\]
In general, we embed \[\AA^{4n-4r}=\AA^{2n-2r}_+\times \AA^{2n-2r}_+\subseteq \AA^{2n}\times \AA^{2m}=\AA^{2n+2m},\] and \[\AA^{2m-2n+4r}=\AA^{2r}_-\times \AA^{2m-2n+2r}_-\subseteq \AA^{2n}\times \AA^{2m}=\AA^{2n+2m},\] giving rise to an embedding
\[\sp_{2n-2r}\times \sp_{m-n+2r}\hra \sp_{m+n}.\]
This allows us to define $\delta_r\in \Sp_{n+m}$ as the image of $\delta_0\times 1_{2m-2n+4r}$.
The stabilizer of $\bP_{n+m}\delta_r$ is then isomorphic to \[\bP_{r}\times^{\Delta\Sp_{n-r}} \bP_{m-n+r}\subseteq \Sp_n\times \Sp_m.\]

As a corollary of \Cref{L:spespher} and \Cref{L:groupglsp} we obtain the following
\begin{lemma}\label{L:jacsp}
Let $\ain{r}{0}{n}$ and set $\bP=\bP_{n-r}\times \overline{\bP_{m-r}}\subseteq \Sp_n\times \Sp_m$. Then $\Sg_{n,m}^r\in\oo_\bP(\Sg_{n,m})$ and
    \[\bS_{\Sg_{n,m}^r}=\M_{n-r,m-r}\times \Sp_r.\] 
\end{lemma}
\subsection{Representation theory}\label{S:repglsp}
We now study the implications for the representation theory of the respective spaces.
\subsubsection{$\M_{n,m}$}
By above observations we have that $S(\M_{n,m}^r)$ is as an $\gl_n\times \gl_m$-representation isomorphic to $\omega_r$,\[\omega_r\coloneq\#-\id_{\bP_{\eta_r}}^{\gl_n\times \gl_m}(\bon_{\bP_{\eta_r}})=\id_{P_{r,n-r}\times P_{m-r,r}}^{\gl_n\times \gl_m}(\abs{n-r}{\frac{r}{2}}\otimes \abs{m-r}{-\frac{r}{2}}\otimes (\bon_r\otimes\abs{r}{\frac{m-n}{2}})S(\Gl_r)),\] \emph{cf.} \cite[§2]{Minguez}.
\begin{lemma}\label{L:locmult}
    The space $\M_{n,m}^r$ is multiplicity-free.
    Moreover, \[\irr_{\M_{n,m}^r}(\gl_n\times \gl_m)=\{\soc(\abs{n-r}{\frac{r}{2}}\times \tau)\otimes \soc(\nu^{\frac{m-n}{2}}\tau^\lor \times\abs{m-r}{-\frac{r}{2}}): \tau\in \irr_r\}.\]
\end{lemma}
\begin{proof}
    We argue by induction on $n+m$ and consider $\pi\otimes\pi'\in \irr_n\times \irr_m$ such that there exists a non-zero morphism $\omega_r\ra\pi\otimes\pi'$.
    Assume first $r=n$. Then the claim follows from Frobenius reciprocity.
    If $r\le 1$, we obtain from \Cref{T:zel} that $\pi$ and $\pi'$ admit a $\rho\coloneq \nu^{\frac{2r-n+1}{2}}$ left- respectively $\rho'\coloneq \nu^{\frac{m-2r-1}{2}}$ right-derivative. We denote by $\tau\coloneq \Dl(\pi)$,  $\tau'\coloneq \mathcal{D}_{\rho',r}^{\mathrm{max}}(\pi')$, $d=\dlm(\pi)$ and $d'=d_{\rho',r,\mathrm{max}}(\pi')$.
    We assume without loss of generality that $d\ge d'$ and obtain by Frobenius reciprocity that 
    \[\ho_{\gl_n\times \gl_m}(\omega_r,\pi\otimes\pi')\subseteq\ho_{\gl_d\times \gl_{n-d}\times \gl_m}(r_{P_{d,n-d}\times \gl_m}(\omega_r), \rho^{\times d}\otimes \tau\otimes \pi').\]
    A quick analysis with the Geometric Lemma gives that the only constituent of $r_{P_{d,n-d}\times \gl_m}(\omega_r)$ admitting a $\gl_d\times \gl_{n-d}$-equivariant map to $\rho^{\times d}\otimes \tau$ is
    \[\id_{P_{1,d-1}\times P_{r-1,n-r}\times P_{m-r,d-1,r-d+1}}^{\gl_d\times \gl_{n-d}\times \gl_m}(\rho\otimes \abs{n-r-1}{\frac{r-1}{2}}\otimes \abs{m-r}{-\frac{r}{2}}\otimes (\bon_r\otimes\abs{r}{\frac{m-n}{2}})S(\Gl_{d-1})\otimes \]\[\otimes (\bon_r\otimes\abs{r}{\frac{m-n}{2}})S(\Gl_{r-d+1})).\]
    Applying Bernstein reciprocity a second time together with \Cref{L:ext} yields that the above Hom-space is embedded in
    \[\ho_{\gl_{n-d}\times \gl_m}(\id_{P_{1,d-1}\times P_{r-1,n-r}\times P_{m-r,d-1,r-d+1}}^{\gl_d\times \gl_{n-d}\times \gl_m}(((\rho')^{\times d-1})^\lor \otimes \abs{n-r-1}{\frac{r-1}{2}}\otimes \]\[\otimes \abs{m-r}{-\frac{r}{2}}\otimes  (\bon_r\otimes\abs{r}{\frac{m-n}{2}})S(\Gl_{r-d+1})), \tau\otimes \pi').\]
    Noting that $ \abs{m-r}{-\frac{r}{2}}\times ((\rho')^{\times d-1})^\lor$ is irreducible, we can use \Cref{L:derrek} to deduce that
    it is a quotient of
    \[(\rho')^{\times d}\times  \abs{m-r}{-\frac{r}{2}}.\]
    Since we assumed that $d'\le d$, we obtain that $d'=d$, and applying Bernstein reciprocity yields that the Hom-space can be embedded into
     \[\ho_{\gl_{n-d}\times \gl_{m-d}}(\id_{P_{1,d-1}\times P_{r-1,n-r}\times P_{m-r,d-1,r-d+1}}^{\gl_d\times \gl_{n-d}\times \gl_m}(\abs{n-r-1}{\frac{r-1}{2}}\otimes \abs{m-r-1}{-\frac{r-1}{2}}\otimes\]\[\otimes (\bon_r\otimes\abs{r}{\frac{m-n}{2}})S(\Gl_{r-d+1})), \tau\otimes \tau').\]
     But observe now that the induced representation is nothing but $\omega_{r-1}.$
     We thus obtain that the above space is bounded by $1$ by the induction hypothesis and hence we proved that it is multiplicity free. Moreover, it is clear that \[\{\soc(\abs{n-r}{\frac{r}{2}}\times \tau)\otimes \soc(\tau^\lor\nu^{\frac{m-n}{2}} \times\abs{m-r}{-\frac{r}{2}}): \tau\in \irr_r\}\subseteq \irr_{\M_{n,m}^r}(\gl_n\times \gl_m).\]
    By the induction hypothesis and the above argument, we know that every $\pi\otimes\pi'$ in the right hand side is of the form
\[\soc(\rho^{\times d}\times \soc(\abs{n-r-1}{\frac{r-1}{2}}\times \tau))\otimes \soc(\soc(\tau^\lor\nu^{\frac{m-n}{2}} \times\abs{m-r-1}{-\frac{r-1}{2}})\times (\rho')^d)\]
     for suitable irreducible and left-$\rho$-reduced $\tau$. The claim now follows from \Cref{L:derrek}.
\end{proof}
For an ordered partition $\alpha=(\alpha_1,\ldots,\alpha_k)$, we define \[\rho_{n,\alpha}\coloneq \abs{n-\alpha_1}{\frac{\alpha_1}{2}}\times\ldots\times \abs{n-\alpha_k}{\frac{\alpha_k}{2}},\]
\[\rho_{m,\alpha}^*\coloneq \abs{m-\alpha_1}{-\frac{\alpha_1}{2}}\times\abs{n-\alpha_2}{-\frac{-m+n+\alpha_2}{2}}\times\ldots\times \abs{n-\alpha_k}{-\frac{-m+n+\alpha_k}{2}}\]
The following is an easy consequence of the Geometric Lemma.
\begin{lemma}\label{L:locgen}
    Let $\pi\in \irr_n,\pi'\in\irr_m$ and chose $\alpha$ with respect to $\lvert\alpha\lvert$ maximal such that we can find $\tau\in\irr_{n-\lvert\alpha\lvert}$ with $\pi\hra \rho_{n,\alpha}\times \tau$.
    
    If $\sum_{i=1}^k n-\alpha_i<n$, set $\bP=\bP_{\lvert\alpha\lvert,n-\lvert\alpha\lvert}\times \overline{\bP_{\lvert\alpha\lvert,m-\lvert\alpha\lvert}}$. Then $\pi\otimes \pi'$ is \[(\M_{n,m}^r,\bP, \rho_{n,\alpha}\otimes \tau\otimes \rho_{m,\alpha}^*\otimes \nu^{\frac{m-n}{2}}\tau^\lor)-\]generic for all $\ain{r}{1}{n}$.
    
    If $\sum_{i=1}^k n-\alpha_i=n$ and $k>1$, set $\bP=\bP_{n-\alpha_1,\alpha_1}\times \overline{\bP_{m-\alpha_1,\alpha_1}}$. Then
    $\rho_{n,\alpha}\otimes \rho_{m,\alpha}^*$ is \[(\M_{n,m}^r,\bP, \abs{n-\alpha_1}{\frac{\alpha_1}{2}}\otimes \abs{n-\alpha_{2}}{\frac{\alpha_{2}}{2}}\times\ldots\times \abs{n-\alpha_k}{\frac{\alpha_k}{2}}\otimes \abs{m-\alpha_1}{-\frac{\alpha_1}{2}}\otimes  \abs{m-\alpha_{2}}{-\frac{\alpha_{2}}{2}}\times\ldots\times \abs{m-\alpha_k}{-\frac{\alpha_k}{2}})\text{-generic}\] for all $\ain{r}{n-\alpha_1+1}{n}$.
\end{lemma}
\begin{lemma}\label{L:basecase}
    The space $\M_{n,m}$ has the lifting property for $\bon_n\otimes\bon_m$.
\end{lemma}
\begin{proof}
We first reduce to the case $n=m$, since for $n<m$, it is easy to see using Bernstein reciprocity for $\gl_m$ that there exists only one orbit in $\M_{n,m}$, namely $\M_{n,m}^0$, which distinguishes $\bon_n\otimes\bon_m$.

For $n=m$ it is again not hard to check with Bernstein reciprocity that the only orbits which are $\bon_n\otimes \bon_m$-distinguishing are $\M_{n,n}^n$ and $\M_{n,n}^0$.
Now the claim is equivalent to the fact that there exists no morphism $T:S(\M_{n,n})\sra \bon_n\otimes\bon_n$ such that the composition
    \[S(\Gl_n)\hra S(\M_{n,n})\sra \bon_n\otimes\bon_n\]
    is non-zero.
    It is clear that the composition is given by integrating compactly supported functions over $\gl_n$ with respect to a fixed Haar-measure. Let $X_k\coloneq \{m\in \M_{n,n}(\Ff): \lvert\det(m)\lvert\le q^{-k}\}$ and $\chi_k$ its characteristic function 
    Then $X_k=\varpi^{-1} X_{k-1}$ and hence
     \[T(\chi_{k})=T(\chi_{{k-1}})=T(\chi')+T(\chi_{k}), \]
     where $\chi'$ is the characteristic function of $\chi_{k-1}\setminus\chi_{k}$. This gives a contradiction, since $T(\chi')$ can by the above observation be computed as a non-vanishing integral.
\end{proof}
Equipped with these insights we are now able to state the following theorem, which in particular gives a new proof for Howe duality in type II.
\begin{theorem}\label{T:matrix}
    The space $\M_{n,m}$ is multiplicity-free, has the lifting property for all irreducible representations and \[\irr_{\M_{n,m}}(\gl_n\times \gl_m)=\{\pi\otimes \mathfrak{T}_n^m(\pi):\, \pi\in \irr_n\}.\]
\end{theorem}
\begin{proof}
This is just a consequence of \Cref{C:liftingproperty}, \Cref{L:locmult}, \Cref{L:locgen}, \Cref{L:basecase} and \Cref{L:localstrugl}. 
\end{proof}
\subsubsection{$\Sg_{n,m}$}
We fix a unitary character $\mu\colon\Ff^\times \ra\C$ and let $\LL_\mu$ be the 
sheaf of \[(P_{n+m},\delta_{P_{n+m}}^{\frac{1}{2}}\mu(\det_{m+n}))-\]equivariant, locally constant functions on $\sg_{n,m}$. We thus obtain that 
\[S(\Sg_{n,m}^r,\restr{\LL_\mu}{\Sg_{n,m}^r})\cong \sigma_r\coloneq \]
\[=\id_{P_r\times P_{m-n+r}}^{\sp_n\times \sp_m}(\mu(\det_{r})\abs{r}{\frac{r+m-n}{2}}\otimes \mu(\det_{m-n+r})\abs{m-n+r}{\frac{r}{2}}\otimes S(\Sp_{n-r})),\]
\emph{cf.} \cite[§1]{KudlaRallis}
\begin{lemma}\label{L:locmultsp}
    The tuple $(\Sg_{n,m}^r,\restr{\LL_\mu}{\Sg_{n,m}^r})$ is multiplicity free and
    \[\irr_{\Sg_{n,m}^r,\restr{\LL_\mu}{\Sg_{n,m}^r}}(\sp_n\times \sp_m)=\]\[=\{\pi\otimes\pi':\pi\otimes\pi'\subseteq \soc(\mu(\det_r)^{-1}\abs{r}{-\frac{r+m-n}{2}}\rtimes \tau)\otimes \soc(\mu(\det_{m-n+r})^{-1}\abs{m-n+r}{-\frac{r}{2}}\rtimes \tau^\lor),\,  \tau\in \irr(\sp_{n-r})\}.\]
\end{lemma}
\begin{proof}
    We proceed analogously to \Cref{L:locmult} and argue by induction on $n+m$. If $r=0$, the claim follows immediately. Otherwise, we proceed like in \Cref{L:locmult} and utilize the $\rho$-derivative of $\pi$ and $\pi'$, where $\rho=\mu^{-1}\lvert-\lvert^{-\frac{n-m+2r-1}{2}}$.
\end{proof}
 For $\alpha=(\alpha_1,\ldots,\alpha_k)$ an ordered partition, we let \[\tau_\alpha\coloneq \mu(\det_{\alpha_1})^{-1}\abs{\alpha_1}{-\frac{\alpha_1+n-m}{2}}\times\ldots\times\mu(\det_{\alpha_k})^{-1}\abs{\alpha_1}{-\frac{\alpha_k+n-m}{2}}\otimes \]\[\mu(\det_{m-n+\alpha_1})\abs{\alpha_1}{\frac{\alpha_1}{2}}\times\mu(\det_{\alpha_2})\abs{\alpha_2}{\frac{\alpha_2+n-m}{2}}\times \ldots\times\mu(\det_{\alpha_k})\abs{\alpha_k}{\frac{\alpha_k+n-m}{2}},\] which by \Cref{T:zel} is irreducible.
\begin{lemma}
   For $\pi\otimes\pi'\in \irr(\sp_n\times\sp_m)$, let $\alpha$ be maximal with respect to $\lvert\alpha\lvert$ such that there exists $\sigma\in \irr(\sp_{n-\lvert\alpha\lvert)}$ with $\pi\otimes\pi'\hra\id_{P_{\lvert\alpha\lvert}\times \overline{P_{\lvert\alpha\lvert}}}^{\sp_n\times\sp_m}(\tau_\alpha\otimes \sigma\otimes\sigma^\lor)$.
   Then $\pi\otimes \pi'\in \irr(\sp_n\times\sp_m)$ is
   \[(\Sg_{n,m}^r,\restr{\LL_\mu}{\Sg_{n,m}^r},\bP,\tau_\alpha\otimes\sigma\otimes\sigma^\lor)-\text{generic}\] for all $\ain{r}{0}{n}$.
\end{lemma}
Again, equipped with our knowledge on the boundary, we are now able to state the following theorem:
\begin{theorem}\label{T:symplectic}
     The space $(\Sg_{n,m},\LL_\mu)$ is multiplicity-free, has the lifting property for all irreducible representations  and \[\irr_{\Sg_{m,n},\LL_\mu}(\sp_n\times \sp_m)\subseteq \]\[\subseteq \{\pi\otimes\pi':\pi\otimes\pi'\subseteq \pi\otimes \mathfrak{T}_n^m(\pi,\mu),\,  \pi\in \irr(\sp_{n})\}.\]
\end{theorem}
\begin{proof}
    By \Cref{C:liftingproperty}, \Cref{L:jacsp}, \Cref{L:locmultsp} and \Cref{T:matrix} we obtain the lifting property, the fact that $(\Sg_{n,m},\LL_\mu)$ is multiplicity free and the fact that for each $\pi\in \irr(\sp_{n})$, $\pi\otimes\pi'$ is distinguished only if $\pi'$ is a direct summand of $\mathfrak{T}_n^m(\pi,\mu)$. 
\end{proof}    
In fact we will prove the following.
    \begin{prop}\label{P:explicit}
        \[\irr_{\Sg_{m,n},\LL_\mu}(\sp_n\times \sp_m)= \{\pi\otimes\pi':\pi\otimes\pi'\subseteq \pi\otimes \mathfrak{T}_n^m(\pi,\mu),\,  \pi\in \irr(\sp_{n})\}.\]
    \end{prop}
To do so, however we need the following slight generalization of \Cref{T:symplectic}.
For $s\in \C$, let
$\LL_{\mu,s}$ be the 
sheaf of $(P_{n+m},\delta_{P_{n+m}}^{\frac{1}{2}}\mu(\det_{m+n})\abs{m+n}{\frac{s}{2}})$-equivariant, locally constant functions on $\sg_{n,m}$.
Moreover, we let $\ain{r}{0}{n}$, fix $\tau\in \irr(\sp_{n-r})$ and define\[\tau_s'\coloneq \mu(\det_r)^{-1}\abs{r}{-\frac{m-n+r+s}{2}}\rtimes\tau\otimes  \mu(\det_{m-n+r})^{-1}\abs{m-n+r}{-\frac{r+s}{2}}\rtimes\tau^\lor,\] which for generic $s$ is irreducible;
We note that for $\ain{r}{0}{n}$ we have that
\[S(\Sg_{n,m}^r,\restr{\LL_{\mu,s}}{\Sg_{n,m}^t})\cong\]
\[\cong\id_{P_r\times P_{m-n+r}}^{\sp_n\times \sp_m}(\mu(\det_{r})\abs{r}{\frac{s+r+n-m}{2}}\otimes \mu(\det_{m-n+r})\abs{s+m-n+r}{\frac{r}{2}}\otimes S(\Sp_{n-r}).)\]
\begin{lemma}\label{L:generics}
    For generic $s\in\C$, $(\Sg_{n,m},\LL_{\mu,s})$ has the lifting property with respect to $\tau_s'$ and
    \[\dim_\C\ho_{\sp_n\times\sp_m}(S(\Sg_{n,m}^r,\restr{\LL_{\mu,s}}{\Sg_{n,m}^r}),\tau_s')=1.\]
    Moreover, if $\pi'\in\irr(\sp_m)$ such that 
    there exists a non-zero morphism
    \[S(\Sg_{n,m}^r,\restr{\LL_{\mu,s}}{\Sg_{n,m}^r})\ra \mu(\det_r)^{-1}\abs{r}{-\frac{m-n+r+s}{2}}\rtimes\tau\otimes \pi',\] then\[\pi'\cong \mu(\det_{m-n+r})^{-1}\abs{m-n+r}{-\frac{r+s}{2}}\rtimes\tau^\lor.\]
    Finally, there exists a, up to scalar, unique morphism
    \[S(\Sg_{n,m},\LL_{\mu,s})\ra\tau_s'\] which is a rational function in $q^s$ for flat sections.
\end{lemma}
\begin{proof}
    The strategy of \Cref{T:symplectic} works equally well here once we showed that     \[\dim_\C\ho_{\sp_n\times\sp_m}(S(\Sg_{n,m}^t,\restr{\LL_{\mu,s}}{\Sg_{n,m}^t}))\le 1\] with equality  only if $t\le r$. Indeed for $t>r$, the non-vanishing of the Hom-space would imply that $\mu\lvert-\lvert^{\frac{s+m-n+1+2r}{2}}$ appears in the cuspidal support of $\tau$, which contradicts the genericity of $s$.

    For $t\le r$, the claim follows by exactly the same argument as in \Cref{L:locmultsp}.

    Finally, by the doubling method, \emph{cf.} \cite{LapRal05}, \cite{GelPiaRal87}, there exists a non-zero map \[f_{s,s'}=S(\Sg_{n,m},\LL_{\mu,s'})\ra\tau_{s}'\] for $s,s'\in \C$ generic, which is a rational function in $q^{s'}$ and $q^s$. Indeed, the claim for $s'$ appears already in \cite[§1.5]{KudlaRallis}, see also \cite[§3.5]{Ato19}. The one for $s$ is an easy consequence of the construction in \cite[§3.5]{Ato19} as $s$ only enters the integral in this construction through the matrix coefficients and if $v_s^\lor$ is a matrix coefficient of $\mu(\det_r)\abs{r}{\frac{m-n+r+s}{2}}\rtimes\tau$ coming from flat sections, it is clear that $v_s^\lor(g)$ is rational in $q^s$ for fixed $g\in \sp_n$. By the Nullstellensatz we can now find a suitable power of $(s-s')$, such that the normalization of $f_{s,s'}$ by this power can be evaluated at $s=s'$, which yields the desired a non-zero map.
\end{proof}
\begin{proof}[Proof of \Cref{P:explicit}]
    We fix $\pi\in\irr(\sp_n)$ and let $\ain{r}{0}{n}$ be maximal such that there exists $\tau\in\irr(\sp_{n-r})$ with $\pi\hra\mu(\det_r)^{-1}\abs{r}{-\frac{m-n+r}{2}}\rtimes\tau$.
    To see that for all summands $\pi'$ of $\mathfrak{T}_n^m(\pi,\mu)$ the representation $\pi\otimes\pi'$ is distinguished, it suffices to construct a surjective map
    \[S(\Sg_{n,m},\LL_\mu)\ra \soc(\tau_0'),\] where
    for $s\in \C$ we set
\[\tau_s' \coloneq\mu(\det_r)^{-1}\abs{r}{-\frac{m-n+r+s}{2}}\rtimes\tau\otimes  \mu(\det_{m-n+r})^{-1}\abs{m-n+r}{-\frac{r+s}{2}}\rtimes\tau^\lor.\]
    One can construct a map \[f_0\colon S(\Sg_{n,m},\LL_\mu)\ra \tau_0'\] as follows. 
By \Cref{L:generics} there exists a non-zero map
\[f_s\colon S(\Sg_{n,m},\LL_{\mu,s})\ra \tau_s'.\]
Hence after renormalizing by a suitable power of $s$, one can extend it to $s=0$, where it gives the desired map $f_0$.
Note that for generic $s$, we saw in the proof of \Cref{L:generics} that $f_s$ restricts to a non-zero map \[S(\Sg_{n,m}^r,\restr{\LL_{\mu,s}}{\Sg_{n,m}^r})\ra\tau_s'.\] By \Cref{L:generics}, this map is up to a scalar unique and hence we can assume it is of the form \[(I_{{\mu(\det_r)\abs{r}{\frac{m-n+r+s}{2}}},\tau}\otimes I_{\mu(\det_{m-n+r})\abs{m-n+r}{\frac{r+s}{2}},\tau^\lor})\circ\]\[\circ\id_{P_r\times P_{m-n+r}}^{\sp_n\times\sp_r}(\bon_{\mu(\det_r)\abs{r}{\frac{m-n+r+s}{2}}}\otimes\bon_{\mu(\det_{m-n+r})\abs{m-n+r}{\frac{r+s}{2}}}\otimes (S(\Sp_{n-r})\ra\tau\otimes\tau^\lor)).\]
Now note that $f_0$ also induces a non-zero map on $S(\Sg_{n,m}^r,\restr{\LL_{\mu}}{\Sg_{n,m}^r})$ by the lifting property and hence the order of the pole of $f_s$ at $s=0$ has to equal the order of the intertwining operator $I_{\bon_{\mu(\det_r)\abs{r}{\frac{m-n+r+s}{2}}},\tau}\otimes I_{\mu(\det_{m-n+r})\abs{m-n+r}{\frac{r+s}{2}},\tau^\lor}$ at $s=0$. In particular, we have by \Cref{T:segint} that  
\[\cosoc(\im(f_0))=\soc(\Pi_0'),\] which is precisely what we wanted to show.
\end{proof}
\subsubsection{Metaplectic groups and other special cases}
Let $G$ be either $\gl_n\times\gl_m$, $\sp_n\times\sp_m$, or a parabolic subgroup of these two groups and let $\pi\colon H\ra G$ be the metaplectic $\mu_2\times \mu_2$ cover of $G$. If $X$ are the $\Ff$-points of a spherical variety of $G$, we can consider $X$ as an $H$-variety via the projection map $\pi$.
Since the results and definitions of \Cref{S:sphericalrep} only depend on the $\Ff$-points of our spherical variety, it is easy to generalize the respective notions to the metaplectic case.
It is now straightforward to generalize the arguments of the above section to obtain the following. Let $n\le m$ be natural numbers.
\begin{theorem}
    The space $\M_{n,m}$ is multiplicity-free, has the lifting property for all irreducible representations and \[\irr_{\M_{n,m}}(\Glt_n\times \Glt_m)=\{\pi\otimes \mathfrak{T}_n^m(\pi):\, \pi\in \irr(\Glt_n),\, \mu_2\text{ acts trivially on }\pi\}.\]
\end{theorem}
Let $\mu$ be a unitary character of $\Gl_1$ and define \[\mu^{(n+m)}\colon\Glt_{n+m}\ra\C^\times,\, (g,\epsilon)\mapsto \epsilon^{n+m}\gamma_\Ff(\det(g,\psi)^{-n-m})\mu(\det(g)), \] see also \Cref{S:2gl}. Note that we can identify $\sg_{n,m}=\tp_{n+m}\bs\mp_{n+m}$, and we let $\LL_\mu$ be the sheaf of $(\mu^{(n+m)}\delta_{\overline{P}_{n+m}}^\frac{1}{2},\tp_{n+m})$-equivariant, locally constant functions on $\sgt_{n,m}$.
\begin{theorem}
     The space $(\sgt_{n,m}',\LL_\mu)$ is multiplicity-free, has the lifting property for all irreducible representations and \[\irr_{\sgt_{m,n},\LL_\mu}(\mp_n\times \mp_m)=\]\[=\{\pi\otimes\pi':\pi\otimes\pi'\subseteq \pi\otimes  \mathfrak{T}_n^m(\pi,\mu^{(n+m)}),\,  \pi\in \irr(\mp_{n}),\, (1,\epsilon)\text{ acts via }\epsilon^{n+m}\text{ on }\pi\}.\]
\end{theorem}
Similarly, we can replace $\gl_n$ by $\gl_n'$ and $\M_{n,m}$ by $\M_{n,m}'$, the space of $n\times m$ matrices with entries in $\DD$. Then it is easy to formulate and to prove analogously to the case $\DD=\Ff$ the following theorem. 
\begin{theorem}
    The space $\mathrm{Mat}_{n,m}'$ is multiplicity-free, has the lifting property for all irreducible representations and \[\irr_{\mathbf{Mat}_{n,m}'}(\gl_n'\times \gl_m')=\{\pi\otimes\soc(\nu^{\frac{m-n}{2}}\pi^\lor\times\abs{n-m}{-\frac{n}{2}}):\, \pi\in \irr_n\}.\]
\end{theorem}
All proofs proceed analogously to the split case.
\subsection{Applications}\label{S:app}
In this section we give two applications of our results.
\subsubsection{Howe duality in type II}
We now quickly recap the well-known theory of the theta-correspondence in type II, \emph{cf.} \cite{Waldspurger}, \cite{MVW},\cite{Kudla1986}, \cite{Minguez}.

Define the Weil-representation of $\gl_n'\times \gl_m'$ as $\omega_{n,m}\coloneq (\abs{m}{\frac{m}{2}}\otimes \abs{n}{-\frac{n}{2}})S(\M_{n,m}')$. The big theta lift of $\pi\in\irr_n$ is defined as the maximal $\pi$-isotypical quotient of $\omega_{n,m}$ and denoted by $\Theta_{n,m}(\pi)$. Then $\Theta_{n,m}(\pi)$ is of finite length and non-zero if $n\ge m$. We denote $\theta_{n,m}(\pi)=\cosoc(\Theta_{n,m}(\pi))$. As a corollary of \Cref{T:matrix} we obtain a new proof of Howe duality in type II.
\begin{theorem}[{Howe duality in type II, \cite{Minguez}}]
    Let $\pi,\pi'\in \irr_n$. Then $\theta_{n,m}(\pi)$ is irreducible and if $\theta_{n,m}(\pi)\cong \theta_{n,m}(\pi')\neq 0$ then $\pi\cong \pi'$.
    If $n\le m$, we can describe $\theta_{n,m}(\pi)$ explicitly as follows.

    Let $\ain{r}{0}{n}$ be minimal such that there exists $\tau\in\irr_r$ with $\pi\hra\abs{n-r}{\frac{n+r}{2}}\times\tau$. Then
    \[\theta_{n,m}(\pi)\cong \soc(\tau^\lor\times \abs{m-r}{-\frac{m+r}{2}}).\]
\end{theorem}
\subsubsection{Local Miyawaki liftings in the Hilbert Siegel case}
We recall the local Miyawaki liftings as introduced in \cite{Ato19}.

Let $\mu$ be a unitary character of $\gl_1$, $\pi\in\irr(\mp_n)$, $m\in \NN$ such that the $\mu_2$-part of $\irr(\mp_n)$ acts via $(\pm 1)^{n+m}$, \emph{i.e.} $\pi$ is genuine if $n+m$ is odd and factors through $\sp_n$ if $n+m$ is even.
We let the group $\mp_n\times\mp_m$ act on $\mp_{n+m}$ as in \Cref{S:sympgra}, however with the twist that on $\mp_m$ we twist the action by conjugation by \[\left(1_{2n},\begin{pmatrix}
    1_m&0\\0&-1_m
\end{pmatrix}\right).\]
We let  $\mu^{(n+m)}\rtimes \bon_1\in \Rep(\mp_{n+m})$, see \Cref{S:2gl}, and restrict this representation via the above action to $\mp_n\times\mp_m$.
Then the maximal $\pi$-isotypical quotient of $\mu^{(n+m)}\rtimes \bon_1\in \Rep(\mp_{n+m})$ is of the form
\[\pi\otimes\mathrm{M}_{\mu}^m(\pi).\]
We set\[\mathrm{Miy}_{\mu}^m(\pi)\coloneq \mathrm{cosoc}(\mathrm{M}_{\mu}^m(\pi))\]
The following is then an easy corollary from the results presented in \Cref{S:repglsp} and \Cref{C:secrep}.
\begin{theorem}
  Let $\pi,\pi'\in\irr(\mp_n)$, $m\in \NN$ such that the $\mu_2$-part of $\irr(\mp_n)$ acts via $(\pm 1)^{n+m}$. If $m\ge n$,   the representation $\mathrm{M}_{\mu}^m(\pi)$ is of finite length and non-zero. In this case, let $\ain{r}{0}{n}$ be maximal such that there exists $\tau\in\irr(\mp_{n-r})$ with $\pi\hra\mu_\psi(\widetilde{\det_r})^{-1}\abs{r}{-\frac{m-n+r}{2}}\rtimes\tau$. Then
  \[\mathrm{Miy}_{\mu}^m(\pi)=\soc(\mu_\psi(\widetilde{\det_{m-n+r}})\abs{m-n+r}{\frac{r}{2}}\rtimes\tau),\] which is of length at most $2$.
  If $n\ge m$ and $\mathrm{Miy}_{\mu}^m(\pi)\cong \mathrm{Miy}_{\mu}^m(\pi')\neq 0$, then $\pi\cong \pi'$.
\end{theorem}
\subsubsection{Computation of Standard $L$-factors}
In this subsection, we highlight how in the case of the space of matrices, the above computations allow one to compute the L-factors of Godement and Jacquet, \emph{cf} \cite{godement1972zeta}, \cite{JL2}.

We recall their definition. Let $\pi\in \Rep_n$ admitting a central character, $s\in \C$, $f$ a matrix coefficient of $\pi$ and $\phi\in S(\M_{n,n}')$. If $\mathrm{Re}(s)>>0$, the integral
    \[Z(s,f,\phi)\coloneq \int_{\gl_n'}f(g)\phi(g)\nu_n(g)\,\mathrm{d} g\]
    converges and defines an element of $\C(q^s)$.

    The ideal $I(\pi)$ generated by $Z(s+\frac{nd-1}{2},f,\phi)$, where $f$ and $\phi$ vary over all possible choices, is a fractional $\C[q^s,q^{-s}]$-ideal of $\C(q^s)$ containing $1$. The generator $L(\pi,s)\in \C[q^{-s}]$ of $I(\pi)$ normalized by the condition \[L(\pi,0)=1\] is called the local \emph{Godement-Jacquet L-factor} associated to the irreducible representation $\pi$. Similarly, one can construct the accompanying $\epsilon$- and $\gamma$-factors, for which we refer to the corresponding literature.

If $\pi$ is irreducible, the map
\[T_\pi\coloneq\restr{Z(s+\frac{nd-1}{2},-,-)L(\pi,s)}{s=-\frac{nd-1}{2}}\colon S(\M_{n,n}')\ra\pi\otimes\pi^\lor\] is by Howe duality the, up to a scalar, unique morphism $S(\M_{n,n}')\ra\pi\otimes\pi^\lor$. 
\begin{lemma}[{\cite[Proposition 2.3]{JL2}}]\label{L:indL}
   Let $\pi_1,\ldots,\pi_k\in \irr$. Then    \[L(\pi_1\times\ldots\times\pi_k,s)=\prod_{i=1}^kL(\pi_i,s).\]
\end{lemma}
The following is clear from the definitions and the fact that $S(\M_{n,n}')$ is multiplicity-free.
\begin{lemma}\label{L:nopole}
    Let $\pi\in \irr_n$. Then the unique non-zero morphism
    \[S(\M_{n,n}')\ra \pi\otimes\pi'\] does not vanish on $S(\Gl_n')$ if and only if $L(\pi,s)$ does not have a pole at $s=-\frac{nd-1}{2}$.
\end{lemma}
Next we prove the following.
\begin{lemma}\label{L:locgenL}
    We use the notation of \Cref{L:locgen} and assume $n=m$.
    Then the pole of $L(\pi,s)$ at $s=-\frac{nd-1}{2}$ equals the pole of $L(\rho_{n,\alpha},s)$ at $s=-\frac{nd-1}{2}$. 
\end{lemma}
\begin{proof}
 By \Cref{L:locgen} it follows that for $\phi\in S(\M_{n,n}')$ with support contained in the matrices 
 \[\left\{\begin{pmatrix}
     g&X\\Y&Z
 \end{pmatrix}:g\in \gl_{n-\lvert\alpha\lvert}',\,X\in \mathrm{Mat}_{n-\lvert\alpha\lvert,\lvert\alpha\lvert}',\, Y\in \mathrm{Mat}_{\lvert\alpha\lvert,n-\lvert\alpha\lvert}',\, Z\in \mathrm{Mat}_{\lvert\alpha\lvert,\lvert\alpha\lvert}'\right\}\]
and of the form
\[\phi\left(\begin{pmatrix}
     g&X\\Y&Z
 \end{pmatrix}\right)=\phi_{1,1}(g)\phi_{2,1}(X)\phi_{1,2}(Y)\phi_{2,2}(Z)\]
for Schwartz functions $\phi_{1,1},\ldots,\phi_{2,2}$ on the respective spaces, $f$ a matrix coefficient of $\pi$, which is induced from a matrix coefficient $f_1$ of $\rho_{n,\lvert\alpha\lvert}$ and $f_2$ a matrix coefficient of $\tau$,
\[T_\pi(f,\phi)=T_{\rho_{n,\lvert\alpha\lvert}}(f_1,\phi_{2,2})\int_{\mathrm{Mat}_{n-\lvert\alpha\lvert,\lvert\alpha\lvert}'\times \mathrm{Mat}_{\lvert\alpha\lvert,n-\lvert\alpha\lvert}'\times \gl_{\lvert\alpha\lvert}'}\phi_{1,1}(g)\phi_{2,1}(X)\phi_{1,2}(Y)\, \mathrm{d}g\,\mathrm{d} X\,\mathrm{d} Y\] for suitable Haar-measures.
From this, it follows easily that the pole of $L(\pi,s)$ at $s=-\frac{nd-1}{2}$ has to have at least the same order as the pole of $L(\rho_{n,\alpha},s)$ at $s=-\frac{nd-1}{2}$. But by \Cref{L:locgen} it also follows that one can choose the above data such that left-hand-side of the equation does not vanish, and thus it is also easy to see that the pole of $L(\pi,s)$ at $s=-\frac{nd-1}{2}$ has at most the same order as the pole of $L(\rho_{n,\alpha},s)$ at $s=-\frac{nd-1}{2}$.
\end{proof}
\begin{lemma}\label{L:trivL}
    The $L$-function of the trivial representation $\mathbf{1}_n$ of $\gl_n'$ has a simple pole at $s=-\frac{nd-1}{2}$.
\end{lemma}
\begin{proof}
    By \Cref{L:nopole} and the fact that in the filtration by ranks of $S(\M_{n,n}')$, $\mathbf{1}_n$ appears as a quotient imply that there is a pole at $s=-\frac{n-1}{2}$.

    Let $K$ be the compact subgroup of $S(\M_{n,n}')$
    consisting of matrices with entries in the ring of integers of $\DD$ and $\chi_K$ its characteristic function. On the one hand, we know that $T_{\mathbf{1}_n}(\chi_K)=\chi_K(0)=1$.
    On the other hand, we obtain the geometric series
    \[Z(s,\chi,1)=\frac{1}{1-q^{-s}}\mathrm{vol}(\{g\in K:\nu_n(g)=1\}),\] which has a simple pole at $s=0$. The claim thus follows.
\end{proof}
We are now able to recover the main (local) theorem of \cite{godement1972zeta}. Let $\varpi$ be a uniformizer of $\Ff$.
\begin{theorem}
Let $\pi\in \irr_n$. Then $L(\pi,s)$ has a pole at $s=-\frac{nd-1}{2}$ if and only if \[T_\pi\colon S(\M_{n,n}')\ra \pi\otimes\pi^\lor\] vanishes on $S(\Gl_n')$.

To be more precise,
let $r$ be maximal such that $T_\pi$ vanishes on $S({\M_{n,n}^{\ge r+1'}})$. Then \[\mathrm{ord}_{s=-\frac{nd-1}{2}}L(\pi,s)=\Lambda(\nu_{n-r}^{\frac{r}{2}},\tau)+1,\]
where $\tau$ is the unique irreducible representation such that $\pi\hra \nu_{n-r}^{\frac{r}{2}}\times\tau$.

Finally, let $\fm\in \Ms$ be the multisegment such that $\L(\fm)\cong \pi$. Then \[L(\pi,s)=\prod_{[a,b]_\rho\in \fm}L(\rho\nr^b,s),\]
where the product is over all segments in $\fm$ (with multiplicity)
and for $\rho$ a cuspidal representation
\[L(\rho,s)=\begin{cases}\frac{1}{1-\chi(\varpi)q^{\frac{1-d}{2}-s}}&\text{ if }\rho\cong \chi\text{ is an unramified character,}\\1&\text{ otherwise.}\end{cases}\]
\end{theorem}
\begin{proof}
    The claim regarding the description in terms of the multisegment follows easily from \Cref{L:trivL}, \Cref{L:indL}, \Cref{L:locgenL}, and \Cref{L:derL}.
    The claim regarding the intertwining operator is a simple consequence of \Cref{T:centralint}.
\end{proof}
\bibliographystyle{amsalpha}
\bibliography{References.bib}
\end{document}